\documentclass[12pt]{article}

\usepackage[top=3.2cm, bottom=3cm, left=3.45cm, right=3.45cm]{geometry}
\usepackage{scalerel}
\usepackage[T1]{fontenc}
\usepackage[utf8]{inputenc}
\usepackage{latexsym}
\usepackage{amsmath,amsthm,amssymb,amsfonts,hyperref}
\usepackage{mathtools}
\usepackage[english]{babel}
\usepackage{tikz}
\usetikzlibrary{patterns, matrix, shapes}
\usepackage{tikz-cd}
\usepackage{booktabs}
\usetikzlibrary{patterns, matrix}
\tikzstyle{NE-lines}=[pattern=north east lines, pattern color=black!45]
\usepackage{soul} 
\usepackage{url}

\usepackage[noblocks]{authblk}

\DeclareMathOperator{\hatt}{hat}
\DeclareMathOperator{\img}{Im}
\newcommand{\etal}{et~al.}

\mathchardef\mhyphen="2D

\DeclareMathOperator{\End}{End}                
\DeclareMathOperator{\Cay}{Cay}                
\DeclareMathOperator{\Sym}{\mathfrak{S}}       
\DeclareMathOperator{\I}{I}                    
\DeclareMathOperator{\Modinv}{\hat{I}}         
\DeclareMathOperator{\Ascseq}{A}               
\newcommand{\Ah}{\hat{\Ascseq}}
\newcommand{\dA}{\Ascseq_d}                    
\DeclareMathOperator{\Modasc}{\hat{\Ascseq}}   
\newcommand{\dModasc}{\Modasc_d}               
\newcommand{\kAscseq}[1]{\Ascseq_{#1}}         
\newcommand{\kModasc}[1]{\hat{\Ascseq}_{#1}}   
\DeclareMathOperator{\wD}{wD}                  
\DeclareMathOperator{\F}{F}                    
\DeclareMathOperator{\irsub}{D^{\scalebox{.5}{$\nearrow$}}} 
\DeclareMathOperator{\drsub}{D^{\scalebox{.5}{$\searrow$}}} 
\DeclareMathOperator{\Bur}{Bur}                
\DeclareMathOperator{\WI}{WI}                  
\DeclareMathOperator{\fishpattern}{\mathfrak{f}} 
\DeclareMathOperator{\burget}{\mathtt{t}}        
\DeclareMathOperator{\sort}{sort}                
\newcommand{\twoplustwo}{\mathbf{2\hspace{-0.2em}+\hspace{-0.2em}2}}   
\newcommand{\beh}{\hat{\be}}
\newcommand{\tah}{\hat{\tau}}
\newcommand{\omh}{\hat{\omega}}
\newcommand{\hatinv}{h}               

\DeclareMathOperator{\Asc}{Asc}          
\newcommand{\dAsc}{\Asc_d}               
\DeclareMathOperator{\asc}{asc}          
\newcommand{\dasc}{\asc_d}               
\DeclareMathOperator{\wDes}{wDes}        
\DeclareMathOperator{\wdes}{wdes}        
\DeclareMathOperator{\nub}{nub}          
\DeclareMathOperator{\identity}{id}      
\DeclareMathOperator{\mind}{dmin}        

\DeclareMathOperator{\rlMinP}{rlMinP}    
\newcommand{\dhat}{\mathrm{hat}_d}
\newcommand{\khat}[1]{\mathrm{hat}_{#1}}
\newcommand{\maxhat}{\khat{\max}}
\DeclareMathOperator{\Act}{Act}

\DeclareMathOperator{\Ascbot}{AscBot}

\DeclareMathOperator{\indx}{ind}         

\DeclareMathOperator{\A}{A}
\newcommand{\D}{\mathtt{D}}
\newcommand{\U}{\mathtt{U}}

\newcommand{\Path}{\mathcal{P}}
\newcommand{\Q}{\mathcal{Q}}
\newcommand{\QQ}{\mathbb{Q}}

\newtheorem{theorem}{Theorem}[section]
\newtheorem{theorem*}{Theorem}[section]
\newtheorem{proposition}[theorem]{Proposition}
\newtheorem{lemma}[theorem]{Lemma}
\newtheorem{corollary}[theorem]{Corollary}
\newtheorem{problem}[theorem]{Problem}
\newtheorem*{openproblem*}{Open Problem}
\newtheorem{question}[theorem]{Question}
\newtheorem{conjecture}[theorem]{Conjecture}
\theoremstyle{definition}
\newtheorem{definition}[theorem]{Definition}

\newtheorem*{remark*}{Remark}

\newtheorem*{example*}{Example}

\newcommand{\ben}{\begin{enumerate}}
\newcommand{\een}{\end{enumerate}}
\newcommand{\ble}{\begin{lem}}
\newcommand{\ele}{\end{lem}}
\newcommand{\bth}{\begin{thm}}
\renewcommand{\eth}{\end{thm}}
\newcommand{\bpr}{\begin{prop}}
\newcommand{\epr}{\end{prop}}
\newcommand{\bco}{\begin{cor}}
\newcommand{\eco}{\end{cor}}
\newcommand{\bcon}{\begin{conj}}
\newcommand{\econ}{\end{conj}}
\newcommand{\bde}{\begin{defn}}
\newcommand{\ede}{\end{defn}}
\newcommand{\bex}{\begin{exa}}
\newcommand{\eex}{\end{exa}}
\newcommand{\barr}{\begin{array}}
\newcommand{\earr}{\end{array}}
\newcommand{\btab}{\begin{tabular}}
\newcommand{\etab}{\end{tabular}}
\newcommand{\beq}{\begin{equation}}
\newcommand{\eeq}{\end{equation}}
\newcommand{\bea}{\begin{eqnarray*}}
\newcommand{\eea}{\end{eqnarray*}}
\newcommand{\bal}{\begin{align*}}
\newcommand{\bce}{\begin{center}}
\newcommand{\ece}{\end{center}}
\newcommand{\bpi}{\begin{picture}}
\newcommand{\epi}{\end{picture}}
\newcommand{\bpp}{\begin{picture}}
\newcommand{\epp}{\end{picture}}
\newcommand{\bfi}{\begin{figure} \begin{center}}
\newcommand{\efi}{\end{center} \end{figure}}
\newcommand{\bprf}{\begin{proof}}
\newcommand{\eprf}{\end{proof}\medskip}

\newcommand{\bsl}{\begin{slide}{}}
\newcommand{\esl}{\end{slide}}
\newcommand{\bfr}{\begin{frame}}
\newcommand{\efr}{\end{frame}}

\newcommand{\eqqed}[1]{$\rule{1ex}{0ex}\hfill{\dil#1}\hfill\qed$}

\newcommand{\hso}[1]{\hspace{-1pt}}

\newcommand{\emp}{\emptyset}

\newcommand{\sbe}{\subseteq}

\newcommand{\alh}{\hat{\al}}

\newcommand{\case}[4]{\left\{\barr{ll}#1&\mbox{#2}\\#3&\mbox{#4}\earr\right.}

\def\<{\langle}
\def\>{\rangle}

\newcommand{\ra}{\rightarrow}

\newcommand{\al}{\alpha}
\newcommand{\be}{\beta}
\newcommand{\ga}{\gamma}
\newcommand{\de}{\delta}
\newcommand{\ep}{\epsilon}

\newcommand{\si}{\sigma}

\DeclareMathOperator{\id}{id}

\newcommand{\dil}{\displaystyle}

\begin{document}
\title{Modified difference ascent sequences\\and Fishburn structures}
\author[1]{Giulio Cerbai\thanks{G.C. is member of the Gruppo Nazionale Calcolo Scientifico--Istituto Nazionale di Alta Matematica (GNCS-INdAM).}}
\author[1]{Anders Claesson}
\author[2]{Bruce E. Sagan}
\affil[1]{Department of Mathematics, University of Iceland,
Reykjavik, Iceland, \texttt{akc@hi.is}, \texttt{giulio@hi.is}.}
\affil[2]{Department of Mathematics, Michigan State University,
East Lansing, MI 48824-1027, USA, \texttt{sagan@math.msu.edu}}
\date{\today\\[10pt]
	\begin{flushleft}
	\small Key Words: ascent sequence, Burge transpose, difference ascent sequence, Fishburn permutation, modified ascent sequence, permutation patterns, weak ascent sequence
	                                       \\[5pt]
	\small AMS subject classification (2020):  05A19  (Primary) 05A05  (Secondary)
	\end{flushleft}}
\maketitle

\begin{abstract}
  Ascent sequences and their modified version play a central role
  in the bijective framework relating several combinatorial
  structures counted by the Fishburn numbers. Ascent sequences are
 positive integer sequences defined by imposing a bound on the
  growth of their entries in terms of the number of ascents
  contained in the corresponding prefix, while modified ascent
  sequences are the   image of ascent sequences under the so-called
  hat map.
  By relaxing the notion of ascent, Dukes and Sagan have recently
  introduced difference ascent sequences. Here we define modified
  difference ascent sequences and study their combinatorial
  properties. Inversion sequences are a superset of the difference
  ascent sequences and we extend the hat map to this domain.
  Our extension depends on a parameter which we specialize to obtain
  a new set of permutations counted by the Fishburn numbers and
  characterized by a subdiagonality property.
\end{abstract}

\thispagestyle{empty} 

\section{Introduction}

Fishburn structures is a collective term for combinatorial objects
counted by the Fishburn numbers. These numbers appear as sequence A22493
in the OEIS~\cite{oeis} and the $n$th Fishburn number is defined as the
coefficient of $x^n$ in the series
\begin{equation*}
  \sum_{n\geq 0}\prod_{k=1}^n\bigl(1-(1-x)^k\bigr).
\end{equation*}
This generating function first appeared 2001 in a paper by
Zagier~\cite{Z:genfun} concerned with bounds on the dimension of the
space of Vassiliev's knot invariants. Eight years later, Bousquet-Mélou,
Claesson, Dukes and Kitaev~\cite{BMCDK:2fp} proved that this series also
enumerates unlabeled interval orders, thus resolving a long standing open
problem. Peter C.\ Fishburn pioneered the study of interval
orders~\cite{Fis70a, Fis70b, Fis85} and it is in honor of him Claesson
and Linusson~\cite{CL11} named the coefficients of Zagier's series.

Bousquet-M\'{e}lou~{\etal}~\cite{BMCDK:2fp} laid the foundation of a
bijective framework relating interval orders, Stoimenow matchings, and
Fishburn permutations, defined by avoidance of a single bivincular
pattern of length three. To link these objects, as well as to count
them, they introduced an auxiliary set of sequences that embody their
recursive structure more transparently, the ascent sequences.
They defined them as certain nonnegative integer sequences whose
growth of their entries is bounded by the number of ascents contained
in the corresponding prefix. Research into Fishburn structures
(sparked by the work of Bousquet-M\'{e}lou~{\etal}) has
blossomed over the last 15 years.  The structures studied are mostly the
ones previously mentioned but also include Fishburn
matrices~\cite{Fis70b, DP:fmat}, descent correcting
sequences~\cite{CL11} and inversion sequences avoiding the covincular
pattern $
\begin{tikzpicture}[scale=0.21, baseline=6.5pt]
\fill[NE-lines] (0,2) rectangle (3,1);
\draw (0.001,0.001) grid (2.999,2.999);
\filldraw (1,2) circle (5pt);
\filldraw (2,1) circle (5pt);
\end{tikzpicture}\,.  $
Recently, Cerbai and Claesson~\cite{CC:ftree} introduced Fishburn
trees and Fishburn covers to obtain simplified versions of the existing
bijections.

The bijection relating ascent sequences with Fishburn permutations is easy
to describe. Ascent sequences encode the recursive construction
of Fishburn permutations by insertion of a new maximum element.
On the other hand, their relation with $(\twoplustwo)$-free posets is better
expressed in terms of a modified version, that is, their bijective image
under the hat map. Roughly speaking, the hat map goes through the ascent tops
of a given ascent sequence; at each step it increases by one all the entries
in the corresponding prefix that are currently greater than or equal to the current ascent top.
Modified ascent sequences interact better with Fishburn trees too, as they
are simply obtained by reading the labels of Fishburn trees with the
in-order traversal. Further, Fishburn trees arise from the max-decomposition
of modified ascent sequences.
In fact, even though they only appeared marginally in the original
paper~\cite{BMCDK:2fp}, modified ascent sequences have recently assumed
a key role in the understanding of Fishburn structures~\cite{CC:tpb,CC:ftree,Ce:pat1,Ce:pat2}.

In 2023, B\'enyi, Claesson and Dukes~\cite{BCD:was} generalized ascent
sequences to weak ascent sequences. They are defined analogously to the
classical case, but (strict) ascents are replaced with weak ascents.
In the spirit of the original framework, the authors provided
bijections with several classes of matrices, posets and permutations.
Among them, weak ascent sequences encode the active site
construction of weak Fishburn permutations, a superset of Fishburn
permutations defined by avoidance of a single bivincular pattern of
length four.

By relaxing the bound on the growth of the rightmost entry further,
that is, by replacing ascents or weak ascents with difference $d$ ascents, Dukes and
Sagan~\cite{DS:das} arrived at $d$-ascent sequences.
This allowed them to generalize the ascent and weak ascent constructions
whose corresponding combinatorial objects now depended on the paramenter $d$.
They also provided natural injections from $d$-ascent sequences to various structures, for example,
permutations avoiding a bivincular pattern of length $d+3$, leaving
the problem of improving these maps to bijections open.
This was done very recently by Zang and Zhou~\cite{ZZ:dperm}, who
introduced what we will call $d$-Fishburn permutations (they used the term $d$-permutations) and proved that their
recursive structure is embodied by $d$-ascent sequences in the same
way as ascent sequences encode Fishburn permutations.

In this paper, we generalize the hat map to $d$-ascent sequences,
obtaining modified $d$-ascent sequences in the process.
We present a recursive construction of modified $d$-ascent sequences
and use it to study their combinatorial properties.
Our framework is in fact more flexible: it extends to inversion sequences,
a superset of $d$-ascent sequences. Further, our definition of the
hat map depends on a parameter whose specific choices lead to
interesting examples.
Fishburn permutations are obtained~\cite{BMCDK:2fp} by applying the
Burge transpose~\cite{CC:tpb} to modified ascent sequences,
and we prove that the same construction holds for modified
$d$-ascent sequences and $d$-Fishburn permutations.
Finally, we initiate the study of pattern avoidance on $d$-Fishburn
permutations.

We start by giving the necessary tools and definitions in
Section~\ref{section_prel}.

In Section~\ref{section_dmod}, we introduce the $d$-hat map and use it to
define the set of modified $d$-ascent sequences. We then provide a
recursive description of modified $d$-ascent sequences and show in
Proposition~\ref{dhat_is_cayley} that they are Cayley permutations
whose set of indices of left-most copies is equal to the $d$-ascent set of the
unmodified sequence.

Section~\ref{section_props_of_dhat} is devoted to the study of
certain properties of the $d$-hat map. Our main result,
Corollary~\ref{dhat_is_bij}, shows that $d$-hat is injective on
modified $d$-ascent sequences. We then consider which statistics
are preserved by $d$-hat in Section~\ref{section_stats}.

In Section~\ref{section_modinv}, we define modified inversion sequences
and the $\maxhat$ map. We show that, under $\maxhat$, a permutation
corresponds bijectively to the inversion sequence recording its
recursive construction by insertion of a new rightmost maximum value.

This approach is pushed further in Section~\ref{sec_subd}. We restrict the
$\maxhat$ map to ascent sequences and weak descent sequences,
characterizing the corresponding sets of permutations as those that are
subdiagonal in a certain sense.

In Section~\ref{sec_dperm}, we prove that $d$-Fishburn permutations
can be obtained as the bijective image of $d$-ascent sequences
under the composition of the $d$-hat map with the Burge transpose,
lifting a classical result by Bousquet-M\'{e}lou~{\etal}~\cite{BMCDK:2fp}
to any $d\ge 0$.

In Section~\ref{sec_pattav}, we enumerate $d$-Fishburn permutations
avoiding $231$ using a bijection with certain Dyck paths and
the cluster method.

Section~\ref{section_final} contains some final remarks and
suggestions for future work.

\section{Preliminaries}\label{section_prel}


For any nonnegative integer number $n$, let $\End_n$ be the set of \emph{endofunctions},
$\al:[n]\to[n]$, where $[n]=\lbrace 1,2,\ldots, n\rbrace$. We sometimes
identify an endofunction $\al$ with the word $\al=a_1\ldots a_n$, where
$a_i=\al(i)$ for each $i\in[n]$. We will use the convention that Greek
letters will usually be used for sequences and the corresponding Roman
letters will be used for their elements
so, for example, $a_i$ will be
the $i$th element of $\al$ unless otherwise indicated.
Let $\End=\cup_{n\geq 0}\End_n$. In general, given a definition of $E_n$
(of elements of size $n$) we let $E=\cup_{n\geq 0}E_n$. Or, conversely,
given a set $E$ whose elements are equipped with a notion of size, we
will denote by $E_n$ the set of elements in $E$ that have size $n$. 

A \emph{Cayley permutation} is an endofunction $\al$ where
$\img\al=[k]$, for some $k\le n$. In other words, $\al$ is a Cayley
permutation if it contains at least one copy of each integer between~$1$
and its maximum element. The set of Cayley permutations of length $n$ is
denoted by $\Cay_n$. For example, $\Cay_1=\left\lbrace 1\right\rbrace$,
$\Cay_2=\left\lbrace 11,12,21\right\rbrace$ and
$$
\Cay_3=
\left\lbrace 111,112,121,122,123,132,211,212,213,221,231,312,321 \right\rbrace.
$$
There is a well-known one-to-one correspondence between ordered
set partitions and Cayley permutations: The Cayley permutation
$\al=a_1\ldots a_n$ encodes the ordered set partition if $[n]$ into subsets $B_1\ldots B_k$ where $k=\max\alpha$
and $i\in B_{a_i}$ for every $i\in [n]$.

An endofunction $\al\in\End_n$ is an \emph{inversion sequence} if
$a_i\le i$ for each $i\in[n]$. We let $\I_n$ denote the set of inversion
sequences of length $n$. For example,
$$
\I_1=\{1\},\quad\I_2=\{11,12\},\quad
\I_3=\left\lbrace 111,112,113,121,122,123\right\rbrace.
$$


Let $\al:[n]\to [n]$ be an endofunction. We call $i\in [n]$ an
\emph{ascent} of $\al$ if $i=1$ or $i\ge2$ and
$$
a_i>a_{i-1}.
$$ 
We define the {\em ascent set} of $\al$ to be
$$
\Asc\al=\{i\in[n] \mid \text{$i$ is an ascent of $\al$}\}
$$ 
and
$$
\asc\al=\#\Asc\al
$$
where, for any set $S$, $\#S$ denotes the cardinality of $S$.
Note that our conventions differ
from some others in the literature in that we are taking the indices
of ascent tops, rather than bottoms, and that $1$ is always an ascent
which is done for the purpose of simplifying the definition of an ascent sequence.
It will sometimes be convenient to order $\Asc\al$ and other similar sets below increasingly to obtain the {\em ascent list}
$$
\Asc\al=(i_1,i_2,\ldots,i_k),
$$
where $k=\asc\al$.  Our notation will not distinguish between the set and its sequence.

From now on, let $\al_i=a_1\ldots a_i$ denote the prefix of $\al$ of length~$i$. Call $\al$ an \emph{ascent sequence} if for all $i\in[n]$ we have
$$
a_i \le 1 + \asc\al_{i-1}.
$$
Note that when $i=1$ we have $a_1\le 1+\asc\epsilon=1$, where~$\epsilon$
denotes the empty sequence. Since the entries of $\al$ are positive
integers, this forces $a_1=1$.
Let $\Ascseq_{0}$ be the set of ascent sequences and let $\Ascseq_{0,n}$
denote the set of ascent sequences of length~$n$. For instance,
$$
\Ascseq_{0,3} = \{111, 112, 121, 122, 123\}.
$$
Clearly, every $\al\in\Ascseq_{0,n+1}$ is of the form $\al=\be a$,
where $\be\in\Ascseq_{0,n}$ and $1\le a\le 1+\asc\be$.
Note that $\Ascseq_{0,n}\subseteq\I_n$. On the other hand, some ascent
sequences are not Cayley permutations, the smallest example of which is
$12124$. Note also that we depart slightly from the original definition
of ascent sequences~\cite{BMCDK:2fp} and other papers on the topic
in that our sequences use the positive, rather than nonnegative, integers.
The reason is that we want to bring all the
families of sequences considered in this paper under the umbrella of
endofunctions of $[n]$ so as to relate them with Cayley permutations and inversion sequences.

The set $\Modasc_0$ of \emph{modified ascent sequences}~\cite{BMCDK:2fp}
is the bijective image of $\Ascseq_0$ under the $\al\mapsto\alh$ mapping,
defined as follows. Given an ascent sequence $\al$, let
$$
M(\al,j)=\al^+,\text{ where }\al^+(i)=a_i+
\begin{cases}
1 & \text{if $i<j$ and $a_i\geq a_{j}$,}\\
0 & \text{otherwise,}
\end{cases}
$$
and extend the definition of $M$ to multiple indices $j_1$,$j_2$,\dots,$j_k$
by
$$
M(\al,j_1,j_2,\dots,j_k)= M\bigl(M(\al,j_1,\dots,j_{k-1}),j_k\bigr).
$$
Then
$$
\alh=M(\al,\Asc\al),
$$
where in this context $\Asc\al$ is the ascent list of~$\al$.
For example, if $\al=121242232$, then $\Asc\al=(1,2,4,5,8)$ and we get
the following where at each stage the entry governing the modification
is underlined while the entries which are modified italicized:
\begin{align*}
  \al              &= 121242232 \\
  M(\al,1)         &= \underline{1}21242232\\
  M(\al,1,2)       &= 1\underline{2}1242232\\
  M(\al,1,2,4)     &= 1\mathbf{3}1\underline{2}42232\\
  M(\al,1,2,4,5)   &= 1312\underline{4}2232\\
  M(\al,1,2,4,5,8) &= 1\mathbf{4}12\mathbf{5}22\underline{3}2=\alh
\end{align*}
More informally, to determine $\alh$, we scan the ascents of $\al$
from left to right; at each step, every element strictly to the left of
and weakly larger than the current ascent top  is incremented by one.
The construction described above can easily be inverted 
since $\Asc\al=\Asc\alh$.
Thus the
mapping $\Ascseq_0 \rightarrow\Modasc_0$ by $\al \mapsto \alh$ is a bijection.

It is easy to turn this into a definition of $\Modasc_0$ which is recursive
by length and will be given later (see Definition~\ref{def_rec_modasc}).
Finally, in~\cite{CC:tpb} it was proved that
\begin{equation}\label{eq_modasc_char}
\Modasc_0=\lbrace \al\in\Cay\mid \Asc\al=\nub\al\rbrace,
\end{equation}
where
$$
\nub\al = \{\min \al^{-1}(j) \mid 1\leq j\leq \max\al \}
$$
is the set of positions of leftmost copies.
The term ``nub'' comes from a Haskell function that removes duplicate
elements from a list, keeping only the first occurrence of each
element. One may also think of nub as a short for ``not used before.''
Interestingly, the nub (under the name ``sequence of first occurrences'')
has recently appeared in an entirely different context as part of the
work of Liang and Sagan~\cite{ls:lclc} on proving log-concavity and
log-convexity results using distributive lattices.

Equation~\eqref{eq_modasc_char}
can be equivalently expressed in terms of Cayley-mesh patterns,
introduced by the first author~\cite{Ce:sort}, as
$$
\Modasc_0\;=\;\Cay\left(
\begin{tikzpicture}[scale=0.50, baseline=19pt]
\fill[NE-lines] (2,0) rectangle (3,3);
\draw [semithick] (0,0.85) -- (4,0.85);
\draw [semithick] (0,1.15) -- (4,1.15);
\draw [semithick] (0,1.85) -- (4,1.85);
\draw [semithick] (0,2.15) -- (4,2.15);
\draw [semithick] (1,0) -- (1,3);
\draw [semithick] (2,0) -- (2,3);
\draw [semithick] (3,0) -- (3,3);
\filldraw (1,2) circle (5pt);
\filldraw (2,1) circle (5pt);
\filldraw (3,2) circle (5pt);
\end{tikzpicture},
\begin{tikzpicture}[scale=0.50, baseline=19pt]
\fill[NE-lines] (1,0) rectangle (2,3);
\fill[NE-lines] (0,0.85) rectangle (0.85,1.15);
\draw [semithick] (0,0.85) -- (3,0.85);
\draw [semithick] (0,1.15) -- (3,1.15);
\draw [semithick] (0,1.85) -- (3,1.85);
\draw [semithick] (0,2.15) -- (3,2.15);
\draw [semithick] (1,0) -- (1,3);
\draw [semithick] (2,0) -- (2,3);
\filldraw (1,2) circle (5pt);
\filldraw (2,1) circle (5pt);
\end{tikzpicture}
\right).
$$
In the above pair of forbidden Cayley-mesh patterns, the
leftmost one indicates an ascent that is not a leftmost
copy; and the one on the right stands for a leftmost
copy that is not an ascent.
Unlike $\Ascseq_0$, not every modified ascent sequence is
an inversion sequence. For instance, the modified ascent sequence
of $\al=1212$ is $\alh=1312$.


Dukes and Sagan~\cite{DS:das} have recently introduced \emph{difference $d$
ascent sequences}. Let $\al\in\End_n$. Given a nonnegative number $d\ge 0$,
we call $i\in[n]$ a \emph{$d$-ascent} if
$i=1$ or $i\ge2$ and
$$
a_i>a_{i-1} - d.
$$
As with ordinary ascents, we have the {\em $d$-ascent set} (or \emph{list})
$$
\dAsc\al=\{ i\in [n] \mid \text{$i$ is a $d$-ascent of $\al$}\}.
$$
and {\em $d$-ascent number}
$$
\dasc\al=\#\dAsc\al.
$$
Note that a $0$-ascent is simply an ascent, while a $1$-ascent is what
is called a {\em weak ascent}:
$$
a_i>a_{i-1}-1\,\iff\, a_i\ge a_{i-1}.
$$
The analogue of the definition of an ascent sequence in the weak case
is as expected.
Call $\al$ a \emph{$d$-ascent sequence} if for all $i\in[n]$ we have
$$
a_i \le 1 + \asc_d\al_{i-1}.
$$
Once again, the above restriction forces $a_1=1$. From now on, denote by
$\kAscseq{d,n}$ the set of $d$-ascent sequences of length $n$.
Clearly, for $d=0$ we recover the set of ascent sequences,
while for $d=1$
we obtain the set of weak ascent sequences of B\'enyi
{\etal}~\cite{BCD:was}.
Note also that $\dAsc\al\subseteq\Asc_{d+1}\al$ for each $d$,
from which the chain of containments
\begin{equation}\label{chain_dasc}
\kAscseq{0}\subseteq\kAscseq{1}\subseteq\kAscseq{2}\subseteq\kAscseq{3}\subseteq\cdots
\end{equation}
follows immediately.

We now connect $d$-ascent sequences and inversion sequences.
\begin{lemma}\label{dasc_iff_invseq}
We have
$$
\I=\bigcup_{d\ge 0}\kAscseq{d}.
$$
\end{lemma}
\begin{proof}
We will prove that each side of the equality is contained in the other.
We first show that $\kAscseq{d}\sbe\I$ for all $d\ge0$ which will give one
of the desired inclusions. If $\al=a_1 \ldots a_n\in\kAscseq{d}$ then $a_1=1$
as required for an inversion sequence. For $i\ge 2$, we have
$$
a_i\le 1+\dasc\al_{i-1}\le 1+(i-1)=i.
$$
Thus $\al\in\I$.

For the other direction, it suffices to show that $\I_n\sbe\kAscseq{n,n}$.
So take $\al=a_1 \ldots a_n\in\I_n$. We have $a_1=1$ as
needed. And for $i\ge2$ we have $a_{i-1}\le i-1 \le n-1$. Hence
$a_i > -1\ge a_{i-1}-n$. Thus every index $i\ge2$ is an $n$-ascent and so
$$
a_i\le i = 1+\asc_n\al_{i-1}
$$
showing that $\al\in\kAscseq{n,n}$.
\end{proof}

We can now calculate the cardinality of certain $\A_{d,n}$.
\begin{proposition}\label{card_d+3}
For all $d\ge0$ we have
$$
\#\A_{d,n}=
\begin{cases}
  n!         & \text{if $n\le d+2$,}\\
  (d+3)!-d!  & \text{if $n=d+3$.}
\end{cases}
$$
\end{proposition}
\begin{proof}
By the previous lemma $\A_{d,n}\sbe\I_n$. Since $\#I_n=n!$, to prove
the first statement of the proposition, it suffices to show that if
$n\le d+2$ then every inversion sequence of length $n$ is a $d$-ascent
sequence.

Let $\al\in\I_n$ where $n\le d+2$.  We claim that for every proper
prefix $\al_i$, $i\le d+1$, we have $\dAsc\al_i=[i]$.
Indeed, consider any element $a_j\in\al_i$. Then, since~$\al$ is an
inversion sequence, 
$$
a_{j-1}\le j-1\le i-1\le d.
$$
Also $a_j\ge1$. So $a_{j-1}-a_j\le d-1<d$, which forces $j\in\dAsc\al_i$,
proving the claim. Now, for all $a_k\in \al$ we have 
$$
a_k\le k = 1 + \asc_d\al_{k-1},
$$
hence $\al$ is a $d$-ascent sequence, as desired.
To prove the second part of the proposition, we just need to show
that when $n=d+3$ there are exactly $d!$ elements of $\I_{d+3}$ which
are not $d$-ascent sequences.

Let $\al=a_1\ldots a_{d+3}$ be such a sequence. We show that the last
three entries of~$\al$ are
$$
a_{d+1},a_{d+2},a_{d+3}=d+1,1,d+3,
$$
while the prefix $\be=a_1\ldots a_d$ can be any inversion sequence of
size~$d$. If we had $\dAsc(\be a_{d+1}a_{d+2})=[d+2]$, then, using an
argument like that of the previous paragraph, we would have~$\al\in\dA$,
which is a contradiction. On the other hand, it follows from the proof
of the first part that $\dAsc(\be a_{d+1})=[d+1]$. So it must be that
$d+2\not\in\dAsc(\be a_{d+1}a_{d+2})$, i.e. $a_{d+2}\le a_{d+1}-d$.
Together with the fact that $a_{d+1}\le d+1$ and $a_{d+2}\ge 1$, this
forces
$$
a_{d+1}=d+1
\quad\text{and}\quad
a_{d+2}=1.
$$
Now, since we assumed that~$\al$ is not a $d$-ascent sequence, but we know that its
prefix $\be a_{d+1}a_{d+2}$ is, it must be that
$$
a_{d+3}>\asc(\be a_{d+1}a_{d+2}) +1 = d+2.
$$
Since $\al\in\I_{d+3}$ we also have $a_{d+3}\le d+3$. It follows that
there is only one choice for the last element of $\al$, namely $a_{d+3}=d+3$.
In the end, we have $\al=\be (d+1)1(d+3)$, where~$\be$ is any inversion
sequence of size~$d$. Since there are $d!$ choices for such~$\be$,
the proposition is proved.
\end{proof}

\section{Modified \texorpdfstring{$d$}{d}-ascent sequences}\label{section_dmod}

We wish to extend the hat map $\al\mapsto\alh$, originally defined on
$\Ascseq_0$, to the set $\A_d$.
Let $\al\in\kAscseq{d}$, for some $d\ge 0$.
The \emph{$d$-hat} of $\al$ is defined as
$$
\dhat(\al)=M\bigl(\al,\dAsc\al\bigr),
$$
where $\dAsc\al$ is the $d$-ascent list of~$\al$.
To illustrate, suppose $d=2$.  Then it is easy to check that 
$\al=12131532$ is a $2$-ascent sequence with
$$
\Asc_2(\al)=(1,2,3,4,6,8).
$$
So, using the same notation as for the example computing $\alh$ in the ascent sequence case,
\begin{align*}
  \al              &= 12131532 \\
  M(\al,1)         &= \underline{1}2131532\\
  M(\al,1,2)       &= 1\underline{2}131532\\
  M(\al,1,2,3)     &= \mathbf{2}\mathbf{3}\underline{1}31532\\
  M(\al,1,2,3,4)   &=2\mathbf{4}1\underline{3}1532\\
M(\al,1,2,3,4,6) &= 24131\underline{5}32\\
  M(\al,1,2,3,4,6,8) &= \mathbf{35}1\mathbf{4}1\mathbf{64}\underline{2} = \hatt_2(\al).
\end{align*}

The $d$-hat map is a natural generalization of the hat map, obtained
by replacing ascents with $d$-ascents. As a special case, we have
$\khat{0}(\al)=\alh$ for each $\al\in\Ascseq_0$.
More generally, to compute $\dhat(\al)$ scan the $d$-ascents of
$\al$ from left to right; at each step, increment by one every element
strictly to the left of and weakly larger than the current $d$-ascent top.
From now on, given $d\ge 0$, we let
$$
\dModasc=\dhat(\dA)
$$
denote the set of \emph{modified $d$-ascent sequences}.

Let us set up some standard notation we shall use throughout the rest
of this paper. We will consider $d$-ascent sequences $\al=\be a$,
where~$a$ is the last letter of~$\al$ and~$\be$ is a $d$-ascent
sequence of size one less than~$\al$. If~$d$ is clear from context,
we let $\alh=\dhat(\al)$ and $\beh=\dhat(\be)$. We also use ``$+$'' as a
superscript that denotes the operation of adding one to the entries
$c\ge a$ of a given sequence, where~$a$ is a threshold determined by
the context. For instance, we denote by $\beh^+a$ the sequence obtained
by adding one to each entry of~$\beh$ that is greater than or equal to~$a$. Clearly, letting~$b$ denote the last letter of~$\be$,
by definition of $\dhat$ we have for every $n\ge 1$ and
$\al\in\kAscseq{d,n}$ 
\begin{equation}
\label{notation1}
\alh=
\case{\beh a}{if $a\le b-d$,}{\beh^+a}{if $a>b-d$.}
\end{equation}
Finally, we will denote
the entries of the above sequences by
\begin{equation}
\label{notation2}
\begin{array}{lllll}
\al=a_1\ldots a_n, && \alh=a'_1\ldots a'_n,\\
\be=b_1\ldots b_{n-1}, && \beh=b'_1\ldots b'_{n-1}, && \beh^+=b''_1\ldots b''_{n-1},
\end{array}
\end{equation}
where~$n$ is the size of~$\al$. The behavior of~$\dhat$ on the last two
letters of $\al\in\dA$ is described more explicitly in the next lemma.

\begin{lemma}\label{final}
Let $\al=a_1\ldots a_n\in\dA$, for some $d\ge 0$ and $n\ge 2$. Let
$\dhat(\al)=\alh=a'_1\ldots a'_n$. Then
$$
a'_{n-1},a'_n=
\begin{cases}
a_{n-1}+1,a_n & \text{if $a_{n-1}-d < a_n \le a_{n-1}$;}\\
a_{n-1},a_n & \text{otherwise.}
\end{cases}
$$
\end{lemma}
\begin{proof}
We use induction on the size of~$\al$. Let $\al=\be a_n$.
The last element of $\alh$ is~$a_{n}$ by definition of $\dhat$.
Similarly, the last letter of $\beh=\dhat(\be)$ is $a_{n-1}$.

Suppose initially that $a_n>a_{n-1}-d$. Then~$n$ is a $d$-ascent and
so $\alh=\beh^+ a_n$. Now, if $a_{n-1}\ge a_n$ then $a'_{n-1}=a_{n-1}+1$
and $\alh$ ends with $a_{n-1}+1,a_n$. Otherwise, if $a_{n-1}<a_n$ then $a_{n-1}$ will not be incremented and~$\alh$ ends with $a_{n-1},a_n$.

Finally, if $a_n\le a_{n-1}-d$ then~$n$ is not a $d$-ascent.
So in this case $\alh=\beh a_n$ and the last two elements
are $a_{n-1},a_n$ again.
\end{proof}

Our next goal is to provide a recursive definition of~$\dModasc$
which does not depend on constructing $\dA$ first.
In the classical case, such a
definition of $\kModasc{0}$ is as follows~\cite{CC:tpb}, where
we use $\alh$ and $\beh$ to denote generic elements of~$\Ah_0$.
Note that this definition permits the computation of an element $\alh$ in $\hat{\A}_0$ directly from a given $\beh$ in $\hat{\A}_0$ without needing to know $\al$ itself.
\begin{definition}\label{def_rec_modasc}
We have $\kModasc{0,0}=\{\epsilon\}$ and $\kModasc{0,1}=\{1\}$.
Let $n\ge 2$. Then every $\alh\in\kModasc{0,n}$ is of one of two forms
depending on whether the last letter forms an ascent with the
penultimate letter:
\begin{itemize}
\item $\alh = \beh a$ \,and\, $1\leq a \leq b$, or
\item $\alh = \beh^+ a$ \,and\, $b< a \leq 1+\asc\beh$,
\end{itemize}
where $\beh\in\kModasc{0,n-1}$ and the last letter of $\beh$ is $b$.
\end{definition}

We wish to highlight a detail that explains why the definition given above
is consistent with letting $\alh=M\bigl(\al,\Asc\al\bigr)$.
Given $\al\in\Ascseq_0$, to compute $\alh$ we increase entries in the
current prefix if and only if we encounter an ascent of $\al$. On the
other hand, Definition~\ref{def_rec_modasc} is stated directly in terms
of the ascents of the modified sequence, i.e. in terms of $\Asc\beh$.
Since it is known~\cite{BMCDK:2fp} that
\begin{equation}\label{eq_asc0_pres}
\Asc\al=\Asc\khat{0}(\al),
\end{equation}
i.e. the ascent set is preserved under the hat map, these two
approaches are in fact equivalent.

In the same spirit, we wish to give a recursive definition of $\dModasc$.
The problem in generalizing Definition~\ref{def_rec_modasc} is that
in general the $d$-ascent set, as well as its cardinality, is not preserved
under $\dhat$. For instance, for $d=1$ we have $\khat{1}(11)=21$ and
$$
\{1,2\}=\Asc_1(11)\neq\Asc_1(21)=\{1\}.
$$
A suggestion for an alternative approach comes from the classical
case $d=0$. Let $\al\in\kAscseq{0}$ and let $\alh=\khat{0}(\al)$.
Then (see~\cite[Theorem 7.3]{CC:tpb} and~\cite[Section 4.1]{BMCDK:2fp},
respectively),
\begin{equation}\label{asc0_max}
\Asc\al=\nub\alh
\quad\text{and}\quad
\asc\al=\max\alh.
\end{equation}
In fact, the corresponding equalities hold for every $d\ge 0$,
as we show in the next proposition.

\begin{proposition}\label{dhat_is_cayley}
Given $d\ge 0$, let $\al\in\dA$ and let $\alh=\dhat(\al)$.
Then $\alh$ is a Cayley permutation with
$$
\dAsc\al=\nub\alh
\quad\text{and}\quad
\dasc\al=\max\alh.
$$
\end{proposition}
\begin{proof}
We use induction on the size of~$\al$. It is easy to see that the
statement holds if~$\al$ has length zero or one.
Let $n\ge 2$ and let $\al\in\kAscseq{d,n}$. As usual, let $\al=\be a$,
for some $\be\in\kAscseq{d,n-1}$ and $1\le a\le 1+\dasc\be$.
By induction, $\beh=\dhat(\be)$ is a Cayley permutation with
$\dAsc\be=\nub\beh$ and $\dasc\be=\max\beh$.
Following the definition of $\dhat$, we consider two possibilities
according to whether or not~$a$ forms a $d$-ascent with the last
letter~$b$ of~$\be$.
\begin{itemize}
\item Suppose $a\le b-d$. Then $\alh=\beh a$. Note that
$\alh\in\Cay_{n}$ since $\beh\in\Cay_{n-1}$ and
$a\le b-d\le\max\be$. Furthermore,
\begin{align*}
\dAsc\al&=\dAsc\be=\nub\beh=\nub\alh\\
\shortintertext{and}
\dasc\al&=\dasc\be=\max\beh=\max\alh.
\end{align*}
\item Suppose $a>b-d$. Then $\alh=\beh^+a$. Once again, it is
easy to see that $\alh\in\Cay_{n}$ as follows.

First note that by the definition of $d$-ascent sequence and induction we have
$$
a\le \asc_d\be +1 =\max\beh  +1.
$$
If $a=\max\beh+1$, 
then
$\beh^+=\beh$ and
$$
\img\alh=\img\beh\cup\{a\}=[\max\beh+1]=[\max\alh].
$$
On the other hand, if $a\le\max\beh$, then the
only gap created in $\beh^+$ (by lifting the entries $c\ge a$)
is filled by~$a$. More formally,
\begin{align*}
\img\alh&=\img\beh^+\cup\{a\}\\
&=\{1,2,\dots,a-1\}\cup\{a+1,a+2,\dots,\max(\beh)+1\}\cup\{a\}\\
&=[\max\beh+1]\\
&=[\max\alh].
\end{align*}
Finally,
\begin{align*}
\dAsc\al&=\dAsc\be\uplus\{n\}\\
&=\nub\beh\uplus\{n\}\\
&=\nub\beh^+\uplus\{n\}\\
&=\nub\alh,\\
\shortintertext{where the last equality follows since~$a$ is a leftmost copy in $\alh$, and}
\dasc\al&=\dasc\be+1=\max\beh+1=\max\alh.
\end{align*}
\end{itemize}
This finishes the proof of the proposition.
\end{proof}


The equality $\dasc\be=\max\dhat(\be)$ proved in
Proposition~\ref{dhat_is_cayley} leads us to the following
recursive definition of modified $d$-ascent sequences, where we use
the same notational conventions as in Definition~\ref{def_rec_modasc}.
\begin{definition}\label{def_rec_dinv}
Let $d\ge 0$ be a nonnegative integer. Let $\kModasc{d,0}=\{\epsilon\}$
and $\kModasc{d,1}=\{1\}$. Suppose $n\ge 2$. Then every
$\alh\in\kModasc{d,n}$ is of one of two forms depending on whether the
last letter forms a $d$-ascent with the penultimate letter:
\begin{itemize}
\item $\alh = \beh a$ \,and\, $1\leq a \leq b-d$, or
\item $\alh = \beh^+ a$ \,and\, $b-d< a \leq 1+\max\beh$,
\end{itemize}
where $\beh\in\kModasc{d,n-1}$ and the last letter of $\beh$ is $b$.
\end{definition}
The reader will immediately realize that the previous definition is
obtained by replacing $\asc\beh$ with $\max\beh$ in
Definition~\ref{def_rec_modasc}. When $d=0$, the two definitions
are equivalent by equation~\eqref{asc0_max}.
Modified $d$-ascent sequences are built recursively by
insertion of a new rightmost entry~$a$, which is at most equal to one plus
the current maximum; the parameter~$d$ determines the cases where the
prefix is rescaled (by adding one to each entry $c\ge a$).
For convenience, the analogous definitions of $\dA$ and $\dModasc$
are illustrated below:
\begin{align*}
(\dA)\quad\al &=\be a,\ 1\le a\le 1+\dasc\be.\\
(\dModasc)\quad\alh &=\begin{cases}
\beh a,\ 1\le a\le b-d;\\
\beh^+a,\ b-d<a\le 1+\max\beh.
\end{cases}
\end{align*}
The equality $\dasc\be=\max\dhat(\be)$ acts as a bridge between the
two definitions.

Let us end this section with a remark. In general, the set $\Ah_{d,n}$
is not included in $\Ah_{d+1,n}$. For instance, we have
\begin{align*}
\Ah_{0,2}&=\{\khat{0}(11),\khat{0}(12)\}=\{11,12\}
\shortintertext{and}
\Ah_{1,2}&=\{\khat{1}(11),\khat{1}(12)\}=\{21,12\}.
\end{align*}

\section{Properties of \texorpdfstring{$d$}{d}-hat}\label{section_props_of_dhat}

We devote this section to the study of several aspects related
to the $d$-hat map just introduced.  Recall that $\hatt_d$ is a map whose domain is the set of $d$-ascent sequences.

Recall from Proposition~\ref{dhat_is_cayley} that
$\nub\dhat(\al)=\dAsc\al$. When $d=0$, using equation~\eqref{eq_asc0_pres}
we obtain the equality
$$
\nub\khat{0}(\al)=\Asc\khat{0}(\al)
$$
characterizing~$\Ah_0$ as a subset of $\Cay$ (see
equation~\eqref{eq_modasc_char}). Since we have established in
Proposition~\ref{dhat_is_cayley} that $\dModasc\subseteq\Cay$ for every
$d\ge 0$, a natural question arises:
\begin{center}
\emph{Is there an analogous equality characterizing $\dModasc$
when $d\ge 1$?}
\end{center}
As mentioned before Proposition~\ref{dhat_is_cayley}, the equality
$\dAsc\dhat(\al)=\dAsc\al$ does not hold for $d\ge 1$. However, we
show in Proposition~\ref{dasc_cont_nub} that one inclusion holds.
First, a simple lemma.

\begin{lemma}\label{dasc_beplus}
Let $\be\in\End$ and let $\be^+$ be the result of increasing every element of $\be$ which is at least $a$ by $1$ for some $a\ge 0$.
Then for all $d\ge0$
$$
\dAsc\be^+\subseteq\dAsc\be
\quad\text{and}\quad
\Asc\be^+=\Asc\be.
$$
\end{lemma}
\begin{proof}
Let $\be=b_1\ldots b_n$ and $\be^+=b'_1\ldots b'_n$, where $b'_i=b_i$,
if $b_i<a$, and $b'_i=b_i+1$, if $b_i\ge a$. Note that the first
position $i=1$ is a $d$-ascent by definition. On the other hand, let $i\ge 2$
and suppose that $i\in\dAsc\be^+$. We show that $i\in\dAsc\be$.
For a contradiction, suppose that $i\notin\dAsc\be$. More explicitly, we
have
\begin{align*}
i\notin\dAsc\be &\iff b_i\le b_{i-1}-d;\\
i\in\dAsc\be^+  &\iff b'_i>b'_{i-1}-d.
\end{align*}
Comparing the two inequalities forces $b'_i=b_i+1$ and $b'_{i-1}=b_{i-1}$. Therefore, we have
$b_{i-1}<a\le b_i$ and
$$
b_i\le b_i+d\le b_{i-1}<a\le b_i,
$$
which gives us the desired contradiction.

By the previous part of the proposition (and since an ascent is a
$0$-ascent), to prove the remaining equality $\Asc\be^+=\Asc\be$ we only
need to show that $\Asc\be^+\supseteq\Asc\be$.
Let $i\in\Asc\be$. If $i=1$, then $i\in\Asc\be^+$. If instead $i\ge 2$,
then $b_i>b_{i-1}$ and thus~$b_i$ will be increased in~$\be^+$ if~$b_{i-1}$
is increased. In any case, we have $b'_i>b'_{i-1}$, hence $i\in\Asc\be^+$.
This completes the proof.
\end{proof}

\begin{corollary}\label{dasc_hat_subset}
Let $d\ge 0$. Suppose that $\al\in\dA$ and let $\alh=\dhat(\al)$.
Then
$$
\dAsc\alh\subseteq\dAsc\al.
$$
\end{corollary}
\begin{proof}
We use induction on the size~$n$ of $\al$, taking the case $n\le 1$ for
granted. Assume $n\ge 2$. Let $\al=\be a$, where $\be\in\kAscseq{d,n-1}$
and $1\le a\le 1+\dasc\be$, and let $\beh=\dhat(\be)$. As usual, we
consider two cases according to whether or not the last letter~$b$
of~$\be$ forms a $d$-ascent with~$a$.

Suppose first that $1\le a\le b-d$. Then $\alh=\beh a$, where
$\dAsc\beh\subseteq\dAsc\be$ by induction. Now by Definition~\ref{def_rec_dinv}
$$
\dAsc\alh=\dAsc\beh\subseteq\dAsc\be=\dAsc\al.
$$

Otherwise, suppose that $b-d<a\le 1+\dasc\be$. Then
$\alh=\beh^+a$. 
Now using Lemma~\ref{dasc_beplus} and induction we have
$$
\dAsc\alh = \dAsc\beh^+\cup\{n\}
\subseteq\dAsc\beh\cup\{n\} 
\subseteq\dAsc\be\cup\{n\} 
=\dAsc\al.
$$
This completes the demonstration.
\end{proof}

Combining Proposition~\ref{dhat_is_cayley} and Corollary~\ref{dasc_hat_subset} immediately gives the following result.

\begin{proposition}\label{dasc_cont_nub}
Let $d\ge 0$.  We have, for any $\alh\in\dModasc$,

\eqqed{
\dAsc\alh\subseteq\nub\alh.
}
\end{proposition}

\subsection{Injectivity of \texorpdfstring{$\dhat$}{d-hat}}

Our next goal is to prove that $d$-hat is injective on~$\dA$ for
every $d\ge 0$. Let $\al\in\I$ be an inversion sequence.
By Lemma~\ref{dasc_iff_invseq},
the quantity
$$
\mind\al=\min\{d\ge 0\mid\al\in\dA\}
$$
is a nonnegative integer for every $\al$. Furthermore, by
equation~\eqref{chain_dasc} if~$\al$ is a $d$-ascent sequence for
some~$d$, then it is a $k$-ascent sequence for every $k\ge d$.
It is natural to study the set
$$
H(\al)=\{\dhat(\al)\mid d\ge\mind\al\}
$$
of all the (meaningful) $d$-hats of $\al$. Note that
$H(\al)\subseteq\Cay$ by Proposition~\ref{dhat_is_cayley}.
Next, we show that $H(\al)$ is finite.

\begin{lemma}\label{size_of_H}
Let $\al\in\I_n$. Then $\mind\al\le n$. Further,
we have $\dhat(\al)=\khat{n}(\al)$ for each $d\ge n-1$. 
\end{lemma}
\begin{proof}
Recall from the proof of Lemma~\ref{dasc_iff_invseq} that
$\I_n\subseteq\kAscseq{n,n}$. The inequality $\mind\al\le n$
follows immediately. Finally, let $d\ge n-1$. Then
$$
\dAsc\al=\Asc_{n-1}\al=[n]
$$
and the equality $\dhat(\al)=\khat{n-1}(\al)$ follows directly from
the definition of $d$-hat.
\end{proof}

By Lemma~\ref{size_of_H}, we have
$$
H(\al)=
\left\lbrace\dhat(\al)\mid\mind\al\le d\le|\al|\right\rbrace,
$$
from which the following corollary is obtained immediately.

\begin{corollary}\label{H_finite}
Let $\al$ be an inversion sequence. Then $H(\al)$ is finite.
\hfill{\qed}
\end{corollary}

Let us now prove that the sets $H(\al)$ are disjoint. The injectivity
of $\dhat$ over $\dA$ will immediately follow as a corollary.

\begin{proposition}\label{Hset_disj}
Let $\al$ and $\si$ be inversion sequences and suppose that
$H(\al)\cap H(\si)\neq\emptyset$. Then $\al=\si$.
\end{proposition}
\begin{proof}
We use induction on the size. The statement clearly holds for inversion
sequences of size $n\le 1$, so suppose $n\ge 2$. Let $\al$ and $\si$
be in $\I_n$, with $H(\al)\cap H(\si)\neq\emptyset$. If
$\ga\in H(\al)\cap H(\si)$, then
$$
\dhat(\al)=\khat{k}(\si)=\ga,
$$
for some $d\ge\mind\al$ and $k\ge\mind\si$. We prove that $\al=\si$.
Denote by $y$ the last letter of $\ga$. Note that the last letters
of~$\al$ and~$\si$ are equal to~$y$ as well. That is, we have
$\al=\be y$ and $\si=\tau y$, where~$\be$ and~$\tau$ denote
the corresponding prefixes of~$\al$ and~$\si$.
We consider two cases, according to whether or not~$y$ is a leftmost
copy in~$\ga$.

Initially, suppose that $n\notin\nub\ga$. Recall by
Proposition~\ref{dhat_is_cayley} that
$$
\dAsc\al=\nub\ga=\Asc_k\si.
$$
In particular, the last position~$n$ is neither a $d$-ascent
in $\al$, nor a $k$-ascent in $\si$. By definition of $\dhat$
and $\khat{k}$, we have, respectively,
\begin{align*}
\ga &=\dhat(\al)=\dhat(\be)y
\shortintertext{and}
\ga &=\khat{k}(\si)=\khat{k}(\tau)y.
\end{align*}
This forces 
$\dhat(\be)=\khat{k}(\tau)$ so that 
$H(\be)\cap H(\tau)\neq\emptyset$.
By induction, we have 
$\be = \tau$ and consequently
$$
\al=\be y=\tau y=\si.
$$

Finally, suppose that $n\in\nub\ga$. The proof is similar
to the previous case, the difference being that here the last position
is a $d$-ascent in~$\al$, as well as a $k$-ascent in~$\si$.
Therefore,
\begin{align*}
\ga &=\dhat(\al)=\dhat(\be)^+y
\shortintertext{and}
\ga &=\khat{k}(\si)=\khat{k}(\tau)^+y,
\end{align*}
and thus $\dhat(\be)^+=\khat{k}(\tau)^+$. Since both $\dhat(\be)^+$
and $\khat{k}(\tau)^+$ are obtained by rescaling entries $c\ge y$,
we have $\dhat(\be)=\khat{k}(\tau)$, and we can finish the proof as in
the previous case.
\end{proof}

\begin{corollary}\label{dhat_is_bij}
For each $d\ge 0$, we have a bijection $\dhat:\dA\to\dModasc$.
\hfill{\qed}
\end{corollary}

\subsection{Statistics preserved by \texorpdfstring{$\dhat$}{d-hat}}\label{section_stats}

Let us now turn our attention to which statistics are preserved
by $d$-hat. Define the {\em weak descent set} of $\al$ to be
$$
\wDes\al =\{i \ge 2 \mid a_i \le a_{i-1} \}
$$
We also say that $i$ is a {\em right-left minimum index} of $\al$ if 
$a_i < a_j$ for all $i<j\le n$. Further, the set of {\em right-left
minima pairs} is
$$
\rlMinP\al =\{(i,a_i) \mid \text{$i$ is a right-left minimum index of $\al$}\}.
$$
The following lemma will be useful.

\begin{lemma}\label{rlminP}
Let $\al=a_1 a_2\ldots a_n=\be a_n$ where $n\ge1$. Then
$$
\rlMinP(\al) = \rlMinP(a_1\ldots a_k) \uplus\{(n,a_n)\}
$$
where $k\ge1$ is the largest right-left minimum index of $\be$ such
that $a_k<a_n$. If no such index exists then we let $k=0$ so that
$\rlMinP(a_1\ldots a_k)=\rlMinP(\emp)=\emp$.
\end{lemma}
\begin{proof}
Consider what happens in passing from $\rlMinP\be$ to
$\rlMinP\al$. Of course, $(n,a_n)$ becomes a right-left minimum
pair in $\rlMinP\al$ since $a_n$ is the last element of the sequence.
Furthermore, any right-left minimum values $a_i$ of $\rlMinP\be$
with $a_i>a_n$ will now have a smaller element to their right and so
it will be removed in the transition to $\rlMinP\al$. The remaining
pairs of $\rlMinP\be$ will be preserved in $\rlMinP\al$.
This is equivalent to our claim.
\end{proof}

\begin{theorem}\label{props_of_H}
Suppose $\al\in\I_n$. We have the following for all $\ga\in H(\al)$:
\begin{enumerate}
\item $\Asc\ga =\Asc\al$.
\item $\wDes\ga = \wDes\al$.
\item $\rlMinP\ga =\rlMinP\al$.
\end{enumerate}
\end{theorem}
\begin{proof}
(a) We induct on $n$ where the case $n\le 1$ is trivial.
Let $\al=\be a$. Pick a $d$ for which $\al$ is a $d$-ascent sequence and
let $\alh=\dhat(\al)$ and $\beh=\dhat(\be)$. We follow our usual conventions~\eqref{notation1} and~\eqref{notation2}
and denote by~$b$, $b'$ and $b''$ the last letter of~$\be$, $\beh$
and~$\beh^+$, respectively. Note that $b'=b$ by Lemma~\ref{final}.
By induction, we have 
$$
\Asc\beh=\Asc\be.
$$
There are now three cases.
First suppose that $a\le b-d$, so that $\alh=\beh a$. From this, the
induction hypothesis, and the fact that $a\le b=b'$ we obtain
$$
\Asc\alh = \Asc(\beh a) = \Asc\beh = \Asc\be = \Asc(\be a) = \Asc \al.
$$
For the next two cases we will have $a>b-d$ so that~$n$ is a $d$-ascent
and $\alh=\beh^+a$. If $a\le b$, then
$$
b''=b'+1=b+1>a.
$$
Thus, using Lemma~\ref{dasc_beplus}, 
$$
\Asc\alh 
= \Asc(\beh^+a) 
= \Asc\beh^+ 
= \Asc\beh 
= \Asc\be   
= \Asc\be a 
= \Asc\al.
$$
Finally, suppose that $a>b$. Then
$$
b''=b'=b<a
$$
and, in a similar manner to the first case,
$$
\Asc\alh
= \Asc\beh^+\cup\{n\}
= \Asc\be\cup\{n\}
= \Asc\al,
$$
proving the first item.

(b) Directly from the definitions, for all inversion sequences~$\al$
of length~$n$ we have $\Asc\al\uplus \wDes\al = [n]$. So this part
follows immediately from (a).

(c) By induction
$$
\rlMinP\beh=\rlMinP\be.
$$
Again, we begin with the case $a\le b-d$ so that $\alh=\beh a$.
By induction and the fact that both~$\al$ and~$\alh$
end in~$a$, we see that the index~$k$ in Lemma~\ref{rlminP} will be the
same for both~$\al$ and~$\alh$. Thus, using the same lemma and the
inductive hypothesis,
\begin{align*}
\rlMinP\alh &= \rlMinP(b'_1\ldots b'_k) \uplus\{(n,a_n)\}\\
&= \rlMinP(b_1\ldots b_k)\uplus\{(n,a_n)\}\\
&= \rlMinP\al.
\end{align*}

Now consider what happens when $a>b-d$ and $\alh=\beh^+a$.
We must relate $\rlMinP\beh$ and $\rlMinP\beh^+$. By the way $\beh^+$
is constructed from $\beh$ we see that every pair $(i,b'_i)\in\rlMinP\beh$
is either replaced by $(i,b'_i+1)\in\rlMinP\beh^+$ if $b'_i\ge a$ or
remains as $(i,b'_i)$ if $b'_i<a$. In particular, $\beh^+ a$ and $\beh a$ will have the same index~$k$ from Lemma~\ref{rlminP}.
Moreover, due to our choice of~$k$,
$$
\rlMinP(b''_1\ldots b''_k)=
\rlMinP(b'_1\ldots b'_k)=
\rlMinP(b_1\ldots b_k).
$$
The proof is now completed in a manner similar to the first case.
\end{proof}

\section{Modified inversion sequences}\label{section_modinv}

Recall from Lemma~\ref{dasc_iff_invseq} that $\I=\bigcup_{d\ge 0}\dA$.
We shall define the set~$\Modinv$ of \emph{modified inversion sequences}
as
\begin{equation}\label{eq_def_modinv}
\Modinv=\bigcup_{d\ge 0}\dModasc.
\end{equation}
An alternative way of arriving at $\Modinv$ is illustrated in the
next result which follows easily from Lemma~\ref{dasc_iff_invseq} and Proposition~\ref{Hset_disj}.

\begin{proposition}
We have the disjoint union

\eqqed{
\Modinv={\biguplus_{\al\in\I}}H(\al).
}
\end{proposition}

By Proposition~\ref{dhat_is_cayley}, modified inversion sequences
are Cayley permutations; that is, $\Modinv\subseteq\Cay$.
Further, by Proposition~\ref{Hset_disj} given any $\ga\in\Modinv$
there is a unique $\al\in\I$ such that $\ga=\dhat(\al)$, for some
$d\ge\mind\al$. Note that such a~$d$ is not unique, but $\al$ is.
This allows us to define a map
$$
\hatinv:\Modinv\to\I
$$
by letting $\hatinv(\ga)$ be the only inversion sequence~$\al$
such that $\ga\in H(\al)$. We wish to describe $\hatinv$
more explicitly. 
First, let us recall~\cite[Section 4.1]{BMCDK:2fp} an algorithm
to define $\hatinv(\ga)$ in the special case where~$\ga$ is
the modified ascent sequence of $\al\in\kAscseq{0}$.
Let $\ga=g_1\ldots g_n$ and let $\Asc\ga=(i_1,\dots,i_k)$.
Then:
\begin{align*}
\mathtt{for}\;i=i_k,&\dots,i_1:\\
\mathtt{for}\;j& =1,\dots,i-1:\\
&\mathtt{if}\;g_j> g_i\;\mathtt{then}\;g_j:=g_j-1.
\end{align*}
The output of the above procedure is the desired ascent sequence~$\al$.
Since $\al\in\kAscseq{0}$, we have $\Asc\al=\Asc\ga$. 
The previous algorithm goes through the $0$-ascents of $\ga$,
from right to left, to determine the cases where the entries in the prefix
need to be decreased. To define $d$-hat, we have replaced $\Asc\al$
with $\dAsc\al$. By Proposition~\ref{dhat_is_cayley}, we have
$\dAsc\al=\nub\ga$. Therefore, by replacing $\Asc\ga$ with
$\nub\ga$ in the algorithm just given we will obtain the desired
generalization of $\hatinv$ to the set $\Modinv$.
Surprisingly, the definition does not depend on $d$.
Instead of writing the algorithm explicitly, we shall give an equivalent,
recursive description of $\hatinv$. Let $\hatinv(\epsilon)=\epsilon$, the
empty sequence, and $\hatinv(1)=1$. Suppose $n\ge 2$ and let
$\ga=g_1\ldots g_n\in\Modinv_n$. Let $\de=g_1\ldots g_{n-1}$. Then
$$
\hatinv(\ga)=
\begin{cases}
\hatinv(\de^-)g_n & \text{if $n\in\nub\ga$;}\\
\hatinv(\de)g_n & \text{otherwise,}
\end{cases}
$$
where $\de^-$ is obtained from $\de$ by decreasing by one each
entry $c>g_n$. The map~$\hatinv:\Modinv\to\I$ defined this way is surjective but not
bijective, and
$$
\hatinv\circ\dhat(\al)=\alpha\quad\text{for every $\al\in\dA$.}
$$
We leave the details to the reader.

To obtain a deeper understanding of~$\Modinv$, it would be interesting
to characterize it as a subset of~$\Cay$ in the same spirit of
equation~\eqref{eq_modasc_char} for~$\kModasc{0}$.
The following proposition is a first step in this direction.

\begin{proposition}
Let $\ga\in\Modinv$. Then $\Asc\ga\subseteq\nub\ga$. Thus,
$$
\Modinv\subseteq\Cay\left(
\begin{tikzpicture}[scale=0.50, baseline=20pt]
\fill[NE-lines] (2,0) rectangle (3,3);
\draw [semithick] (0,0.85) -- (4,0.85);
\draw [semithick] (0,1.15) -- (4,1.15);
\draw [semithick] (0,1.85) -- (4,1.85);
\draw [semithick] (0,2.15) -- (4,2.15);
\draw [semithick] (1,0) -- (1,3);
\draw [semithick] (2,0) -- (2,3);
\draw [semithick] (3,0) -- (3,3);
\filldraw (1,2) circle (5pt);
\filldraw (2,1) circle (5pt);
\filldraw (3,2) circle (5pt);
\end{tikzpicture}\right).
$$
\end{proposition}
\begin{proof}
Since $\ga\in\Modinv$, there exist $\al\in\I$ and $d\ge\mind\al$
such that $\ga=\dhat(\al)$. In particular,
$$
\Asc(\ga)\subseteq\dAsc(\ga)\subseteq\nub\ga,
$$
where the last set containment is Proposition~\ref{dasc_cont_nub}.
\end{proof}

\subsection{Maximal \texorpdfstring{$d$}{d}-hat}\label{section_nhat}

Recall from Lemma~\ref{size_of_H} that $\mind\al\le n$ for each
$\al\in\I_n$. By proposition~\ref{Hset_disj}, for each $n\ge 0$ we
have an injection
\begin{align*}
\khat{n}:\I_n\;&\longrightarrow\;\Modinv_n\\
\al\;&\longmapsto\;\khat{n}(\al).
\end{align*}
Since---again by Lemma~\ref{size_of_H}---applying $d$-hat gives the
same result for every $d\ge n-1$, we will call \emph{max-hat} the injection
\begin{align*}
\maxhat:\I\;&\longrightarrow\;\Modinv\\
\al\;&\longmapsto\;\khat{|\al|-1}(\al).
\end{align*}
The main goal of this subsection is to prove that $\maxhat$ maps
$\I$ bijectively to~$\Sym$. Namely, we show that $\maxhat(\al)$ is
the permutation whose recursive construction by insertion of a new
rightmost entry is encoded by~$\al$.

We start with a simple lemma.

\begin{lemma}\label{lemma_nhat}
Let $\al\in\I_n$. Suppose that $\al=\be a$, for some $\be\in\I_{n-1}$
and $1\le a\le n$. Then
$$
\maxhat(\al)=\maxhat(\be)^+a.
$$
\end{lemma}
\begin{proof}
We have:
\begin{align*}
\maxhat(\al) &=\khat{n}(\be a)\\
&=\khat{n-1}(\be)^+a
  &&\text{(since $\Asc_{n}\al=[n]$)}\\
&=\maxhat(\be)^+a
  &&\text{(by the definition of $\maxhat$).}
\end{align*}
This concludes the proof.
\end{proof}

\begin{lemma}\label{max_hat}
Let $\al\in\I$. Then $\maxhat(\al)\in\Sym$.
\end{lemma}
\begin{proof}
We use induction on the size~$n$ of $\al$, where the case $n\le 1$
is easy to prove. Let $n\ge 2$. Let $\al=\be a$, for some $\be\in\I_{n-1}$
and $a\in [n]$. By Lemma~\ref{lemma_nhat}, we have
$\maxhat(\al)=\maxhat(\be)^+a$, which is clearly a permutation
since $\maxhat(\be)\in\Sym_{n-1}$ by induction. 
\end{proof}

\begin{corollary}
\label{HatMatCor}
We have a size-preserving bijection $\maxhat:\I\to\Sym$.
\end{corollary}
\begin{proof}
By Proposition~\ref{Hset_disj} and Lemma~\ref{max_hat}, the map
$\khat{n}:\I_n\to\Sym_n$ is injective for every $n\ge 0$. The theorem
follows since it is well known that $\I_n$ and $\Sym_n$ are equinumerous.
\end{proof}

The behavior of $\maxhat$ on $\I$ can be summarized by saying
that~$\al$ encodes the construction of $\maxhat(\al)$ by insertion
of a new rightmost entry. More specifically, when we modify $\al$
under $\maxhat$, at each step we increase by one all the entries in
the current prefix that are greater than or equal to the current
rightmost one. This step-by-step process is illustrated below for
$\al=1224315$:
\begin{align*}
1       &\longmapsto 1\\
12      &\longmapsto 1\underline{2}\\
122     &\longmapsto 1\mathbf{3}\underline{2}\\
1224    &\longmapsto 132\underline{4}\\
12243   &\longmapsto 1\mathbf{4}2\mathbf{5}\underline{3}\\
122431  &\longmapsto \mathbf{25364}\underline{1}\\
1224315 &\longmapsto 2\mathbf{6}3\mathbf{7}41\underline{5}\;=\;
\maxhat(\al)\\
\end{align*}

We end this section with a simple remark. A \emph{flat step}
in $\al=a_1\ldots a_n\in\End$ is a pair of consecutive equal entries
$a_i=a_{i+1}$. Let $\al\in\I$ and let $\ga=\maxhat(\al)$.
It is easy to see that $a_i=a_{i+1}$ is a flat step in~$\al$ if and
only if in~$\ga$ we have $g_i >g_{i+1}$ and no entries $g_j$, $j<i$,
satisfy $g_{i+1}<g_j<g_i$. Define the mesh pattern $\mathfrak{a}$
accordingly as
$$
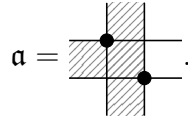
\mathfrak{a}=\begin{tikzpicture}[scale=0.50, baseline=19pt]
\fill[NE-lines] (1,0) rectangle (2,3);
\fill[NE-lines] (0,1) rectangle (2,2);
\draw [semithick] (0,1) -- (3,1);
\draw [semithick] (0,2) -- (3,2);
\draw [semithick] (1,0) -- (1,3);
\draw [semithick] (2,0) -- (2,3);
\filldraw (1,2) circle (5pt);
\filldraw (2,1) circle (5pt);
\end{tikzpicture}.
$$
The next proposition follows immediately.

\begin{proposition}\label{flatsteps}
The map $\maxhat$ restricts to a bijection between inversion
sequences with no flat steps and permutations avoiding~$\mathfrak{a}$.
\hfill{\qed}
\end{proposition}

\section{Subdiagonal permutations}\label{sec_subd}

Recall from Subsection~\ref{section_nhat} that an inversion sequence~$\al$
encodes the construction of $\maxhat(\al)$ by insertion of a new
rightmost entry. In this section, we restrict $\maxhat$ to the
set of ascent sequences and characterize the resulting set of permutations
which have the following subdiagonal property.

Any permutation $\pi\in\Sym$ factors 
uniquely by maximal increasing runs
as $\pi=B_1B_2\ldots B_k$, where $k=n+1-\asc\al$. 
We say that
$\pi$ is
\begin{align*}
&\text{\emph{ir-superdiagonal}}, && \text{if $c\ge i$ for each $c\in B_i$};\\
&\text{\emph{ir-subdiagonal}}, && \text{if $c\le n+1-i$ for each $c\in B_i$},
\end{align*}
where the prefix ``ir'' denotes that~$\pi$ is decomposed by ``increasing
runs''. Clearly, two analogous notions are obtained by replacing
maximal increasing runs with maximal decreasing runs; that is, if
$\pi=C_1\ldots C_k$, where now the blocks $C_i$ are maximally decreasing,
we say that~$\pi$ is
\begin{align*}
&\text{\emph{dr-superdiagonal}}, && \text{if $c\ge i$ for each $c\in C_i$};\\
&\text{\emph{dr-subdiagonal}}, && \text{if $c\le n+1-i$ for each $c\in C_i$}.
\end{align*}
It is easy to see that $\pi$ is ir-subdiagonal if and only if its
complement is dr-superdiagonal; similarly, it is dr-subdiagonal if and only
if its complement is ir-superdiagonal.
So it is no restriction to only consider subdiagonal permutations, denoted by
\begin{align*}
\irsub &=\{\pi:\pi\text{ is ir-subdiagonal}\};\\
\drsub &=\{\pi:\pi\text{ is dr-subdiagonal}\}.
\end{align*}

In the following two subsections, we shall prove that
$$
\maxhat(\kAscseq{0})=\irsub
\quad\text{and}\quad
\maxhat(\wD)=\drsub,
$$
where $\wD$ denotes the set of weak descent sequences, defined later.
As a result of what was observed
in Subsection~\ref{section_nhat}, ascent
sequences encode the recursive construction of ir-subdiagonal
permutations by successive insertions of a new rightmost entry.
And weak descent sequences encode dr-subdiagonal permutations
in the same way. This construction is reminiscent of the way
ascent sequences encode Fishburn permutations~\cite{BMCDK:2fp},
the difference being that in the case of Fishburn permutations a
new maximum is inserted at each step. Note that we have not been able to find
bivincular patterns characterizing $\irsub$ and $\drsub$.
Finally, we define an isomorphism between
two generating trees for weak descent sequences and \emph{primitive
ascent sequences}, defined as those ascent sequences that have no
flat steps.

\subsection{ir-subdiagonal permutations}\label{sec_irsub}

Throughout this section, we let $\pi=B_1\ldots B_k$ be the decomposition
of a given permutation~$\pi$ into maximal increasing runs.
If $c$ is an entry of $B_j$, $1\le j\le k$, we let $\indx_{\pi}(c)=j$
denote the index of the block of $\pi$ that contains~$c$.
Letting $\pi=p_1\ldots p_n$
and $\pi_i=p_1\ldots p_i$, it is easy to see that
$$
\pi\in\irsub\iff p_i\le |\pi|+1-\indx_{\pi}(p_i)
$$
for each $i$, where
\begin{equation}\label{eq_indx}
\indx_{\pi}(p_i)=i+1-\asc\pi_i.
\end{equation}
The next lemma shows that ir-subdiagonal permutations and ascent
sequences share a similar recursive structure.

\begin{lemma}\label{subd_rec}
Let $\pi=p_1\ldots p_n$ and $a\in [n+1]$. Then
$$
\pi^+a\in\irsub \text{ if and only if } \pi\in\irsub \text{ and } a\le 1+\asc\pi.
$$
\end{lemma}
\begin{proof}
We will prove the reverse implication as the forward one is similar.
We start by showing that entries in the prefix $\pi^+$ satisfy the
subdiagonality constraint in $\pi^+a$.
Suppose that $\pi=B_1 \ldots B_k$ is the increasing run decomposition of $\pi$ so that 
$\max B_i \le n-i+1$ for $i\in[k]$ since $\pi\in\irsub$.  It follows that 
$\pi^+=B_1^+ \ldots B_k^+$ is the increasing run decomposition of $\pi^+$ and
$$
\max B_i^+ \le \max B_i +1\le (n+1) - i + 1.
$$
So $\pi^+$ satisfies the ir-subdiagonal restrictions as the initial factor of $\pi^+ a$.

There remains to show that $a$ also satisfies the  ir-subdiagonal restriction.  There will be two cases depending on its size relative to $p_n$.
Suppose first that $a\le p_n$.  Then we have the  increasing run decomposition $\pi^+ a = B_1^+ \ldots B_k^+ B_{k+1}$ where $B_{k+1}=a$,
and the desired inequality is $a\le (n+1)-(k+1)+1$.  But since $p_n\in B_k$ we have $p_n\le n-k+1$ which, combined with $a\le p_n$, finishes this case.
If $a>p_n$ then our increasing run decomposition is $\pi^+ a = B_1^+ \ldots B_{k-1}^+ B_k'$ where $B_k' = B_k^+ a$.  Since $B_k$ is the last run of $\pi$ we have $k=n-\asc\pi+1$.
In this case we want $a\le (n+1)-k+1$.  But by the equation for $k$ and inequality for $a$ assumed in this direction
$$
(n+1)-k+1 =(n+1)-(n-\asc\pi+1)+1=1 + \asc\pi\ge a
$$
which finishes the proof.
\end{proof}


\begin{theorem}\label{subD_iff}
Let $\al\in\I$ and let $\alh=\maxhat(\al)$. Then
$$
\al\in\kAscseq{0}\iff\alh\in\irsub.
$$
Therefore, $\maxhat$ restricts to a bijection from $\kAscseq{0}$
to $\irsub$.
\end{theorem}
\begin{proof}
We use induction on the size of $\al$ where the result is clear for
size at most one. Let $\al=\be a$,
for some $\be\in\I_n$. By Lemma~\ref{lemma_nhat}, we have
$\alh=\beh^+a$. Using induction, we have that $\be\in\kAscseq{0}$
if and only if $\beh\in\irsub$. Now using Theorem~\ref{props_of_H}
we have
$$
\al\in\kAscseq{0} \iff\  \be\in\kAscseq{0} \text{ and } a\le 1+\asc\be =1+\asc\beh.
$$
But, by the lemma just proved, the inequality is equivalent
to $\alh=\beh^+a\in\irsub$ as desired.
\end{proof}

We have just shown that  the set $\irsub$
of ir-subdiagonal permutations is the bijective image of the
set~$\kAscseq{0}$ of ascent sequences under $\maxhat$.
Furthermore, by Proposition~\ref{flatsteps} primitive ascent sequences are in bijection with
ir-subdiagonal permutations avoiding~$\mathfrak{a}$.
The next corollary follows immediately.

\begin{corollary}
For each $n\ge 0$, the number of ir-subdiagonal permutations of size~$n$
is equal to the $n$th Fishburn number, that is, the number of ascent sequences of length $n$. Furthermore, the number of
ir-subdiagonal permutations avoiding~$\mathfrak{a}$ is equal to the
number of primitive ascent sequences (see also A138265~\cite{oeis}).
\hfill\qed
\end{corollary}

\subsection{dr-subdiagonal permutations and weak descent sequences}\label{sec_drsub}

Recall that the set of weak descents of $\al\in\End$ is
$$
\wDes\al =\{i\ge 2\mid a_i\le a_{i-1} \}
$$
Note that $[n]=\wDes\al\uplus\Asc\al$ for every $\al\in\End_n$;
that is, every $i\in [n]$ is either a weak descent or a strict
ascent. The set $\wD$ of \emph{weak descent sequences} is defined as
$$ 
\wD_n=
\{\al\in\I_n\mid a_1=1 \text{ and } a_i\le 1+\wdes\al_{i-1}\text{ for each }i\in[n]\},
$$
where $\wdes\al=|\wDes\al|$.

The next result is a counterpart of Theorem~\ref{subD_iff} and
states that~$\al\in\I$ is a weak descent sequence if and only if
$\maxhat(\al)$ is dr-subdiagonal. Its proof is obtained by simply
replicating the steps of Lemma~\ref{subd_rec} and
Theorem~\ref{subD_iff}, and  is thus omitted.
%

\begin{theorem}
Let $\pi=p_1\ldots p_n$ and $a\in[n+1]$.  Then
$$
\pi^+a\in\drsub \text{ if and only if }
\pi\in\drsub \text{ and }
a\le1+\wdes\pi.
$$
Furthermore, if $\al\in\I$ and $\alh=\maxhat(\al)$,
then
$$
\al\in\wD\iff\alh\in\drsub
$$
and $\maxhat$ restricts to a bijection from $\wD$ to $\drsub$.
\hfill{\qed}
\end{theorem}

We wish to prove that weak descent sequences (and thus dr-subdiagonal
permutations) are equinumerous with primitive ascent sequences.
A generating tree for ascent sequences is encoded by the following
generating rule, where the pair $(a,\ell)$ keeps track of the number
of ascents, $a$, and the last letter, $\ell$:
$$
\begin{cases}
\text{Root: }(1,1)\\
(a,\ell)\longrightarrow (a,1)(a,2)\ldots(a,\ell-1)(a,\ell)
(a+1,\ell+1)(a+1,\ell+2)\ldots(a+1,a+1).
\end{cases}
$$
The above rule encodes the standard construction of ascent sequences
by insertion of a new rightmost entry. The root $(1,1)$ corresponds
to the only ascent sequence of size one, namely the single letter
word~$1$. Further, if $\al\in\kAscseq{0}$ has $a$ ascents and last
letter~$\ell$, then it produces $a+1$ children by insertion of a
new rightmost entry $i\in [a+1]$. If $i\le\ell$, then the number of
ascents remains the same; otherwise, if $i>\ell$, then a new ascent
is created. To obtain a generating rule for primitive ascent
sequences, we remove the child $(a,\ell)$ corresponding to a
flat step and obtain:
$$
\Omega:\begin{cases}
\text{Root: }(1,1)\\
(a,\ell)\longrightarrow (a,1)(a,2)\ldots(a,\ell-1)
(a+1,\ell+1)(a+1,\ell+2)\ldots(a+1,a+1).
\end{cases}
$$
A generating rule for weak descent sequences that uses the number of
weak descents~$w$ and the last letter~$u$ as parameters is now
obtained similarly as:
$$
\Theta:\begin{cases}
\text{Root: }(0,1)\\
(w,u)\longrightarrow (w+1,1)(w+1,2)\ldots(w+1,u)
(w,u+1)(w,u+2)\ldots(w,w+1).
\end{cases}
$$
In this case, inserting $i\in [w+1]$ creates a new weak descent if and
only if $i\le u$.

To show that primitive ascent sequences and weak
descent sequences are equinumerous, we shall give a bijection between
the generating trees encoded by the rules $\Omega$ and~$\Theta$.
Namely, we show that $\Omega$ and $\Theta$ are equivalent under
the linear transformation
\begin{equation}\label{genrule}
\begin{cases}
w=a-1\\
u=a-\ell+1
\end{cases}
\iff\;\;\,
\begin{cases}
a-w=1\\
\ell+u=a+1.
\end{cases}
\end{equation}
Indeed, the root $(a,\ell)=(1,1)$ is mapped to $(w,u)=(0,1)$.
Further, assume that $(a,\ell)$ is mapped to $(w,u)$, i.e. that $w=a-1$ and $u=a+1-\ell$.
Then the children of $(a,\ell)$ are mapped bijectively to the children
of $(w,u)$, since
\begin{align*}
(a,1)        &\,\mapsto\, (a-1,a)=(w,w+1)\\
(a,2)        &\,\mapsto\, (a-1,a-1)=(w,w)\\
&\;\;\vdots\\
(a,\ell-1)   &\,\mapsto\, (a-1,a-\ell+2)=(w,u+1)\\
(a+1,\ell+1) &\,\mapsto\, (a,a-\ell+1)=(w+1,u)\\
(a+1,\ell+2) &\,\mapsto\, (a,a-\ell)=(w+1,u-1)\\
&\;\;\vdots\\
(a+1,a+1)    &\,\mapsto\, (a,1)=(w+1,1).
\end{align*}
As a result, we obtain a bijection between the generating tree of
primitive ascent sequences,
encoded by $\Omega$, and the generating tree of $\wD$,
encoded by $\Theta$. The next result follows immediately.

\begin{corollary}
For each $n\ge 0$, the number of weak descent sequences of size~$n$
is equal to the number of primitive ascent sequences
of size~$n$.
\hfill{\qed}
\end{corollary}


\section{Difference Fishburn permutations}\label{sec_dperm}

Prompted by a question in~\cite{DS:das}, Zang and Zhou~\cite{ZZ:dperm}
have recently introduced \emph{$d$-Fishburn permutations}, defined as follows.
Fix $d\ge 0$ and let $\pi=p_1\ldots p_n$ be a permutation of $\Sym_n$,
with $n\ge 1$. 
We denote by $\pi^{(k)}$ the subsequence of $\pi$ which contains the
elements $[k]$. For example,  if $\pi=641523$ then $\pi^{(4)}=4123$.
The following procedure defines the \emph{$d$-active elements} of $\pi$:
\begin{itemize}
\item Set $1$ to be a $d$-active element.
\item For $k=2,3,\dots,n$, let $k$ be $d$-inactive if $k$ is to the
left of $k-1$ in $\pi$ and there exist at least $d$ elements of $\pi^{(k)}$ 
between $k$ and $k-1$ that are $d$-active.
Otherwise, $k$ is said to be $d$-active.
\end{itemize}
Returning to our example $\pi=641523$ with $d=2$ we compute the $d$-active elements as follows, where such elements are set in boldface.
By the initial condition
$$
\pi^{(1)} = {\bf 1}.
$$
Next
$$
\pi^{(2)} = {\bf 12}.
$$
since $2$ is to the right of $1$ and so will be active.  Similarly
$$
\pi^{(3)} = {\bf 123}.
$$
Now
$$
\pi^{(4)} = 4{\bf123}
$$
with $4$ not active since the number of active elements between it and $3$ is $2\ge d$.  Clearly
$$
\pi^{(5)} = 4{\bf1523}.
$$
Finally
$$
\pi=\pi^{(6)} = {\bf6}4{\bf1523}
$$
where $6$ is active since the number of active elements between it and $5$ is $1<d$.

Let $\Act_d\pi$ denote the set of $d$-active elements of $\pi$. 
Furthermore, denote by $\Ascbot\pi$ the set
$$
\Ascbot\pi=\{ p_i\in[n-1] \mid p_i<p_{i+1}\}
$$
of \emph{ascent bottoms} of $\pi$. 
Note that these are elements of $\pi$ rather than positions.
Then, $\pi$ is said to be a
\emph{$d$-Fishburn permutation} if
$$
\Ascbot\pi\subseteq\Act_d\pi,
$$
and we denote by $\F_d$ the set of $d$-Fishburn permutations.
Recall that Fishburn permutations~\cite{BMCDK:2fp} are defined
as those permutations avoiding the bivincular pattern
$$
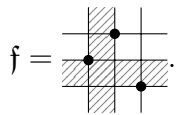
\fishpattern=
\begin{tikzpicture}[scale=0.35, baseline=17.5pt]
\fill[NE-lines] (1,0) rectangle (2,4);
\fill[NE-lines] (0,1) rectangle (4,2);
\draw (0.001,0.001) grid (3.999,3.999);
\filldraw (1,2) circle (5pt);
\filldraw (2,3) circle (5pt);
\filldraw (3,1) circle (5pt);
\end{tikzpicture}.
$$

We wish to give an alternative definition of $d$-Fishburn
permutations that is reminiscent of the classical case $d=0$.
We say that a permutation~$\pi$ contains the \emph{$d$-Fishburn pattern}, $\fishpattern_d$, if it contains an occurrence
$p_ip_{i+1}p_j$ of $\fishpattern$ where $p_i$ is $d$-inactive.
The other two elements $p_{i+1}$ and $p_j$ can be either $d$-active
or $d$-inactive.
With a slight abuse, we will use the suggestive notation
$\Sym(\fishpattern_d)$
to denote the set of permutations that do not
contain~$\fishpattern_d$. 

\begin{proposition}
For every $d\ge 0$,
$$
\F_d=\Sym(\fishpattern_d).
$$
\end{proposition}
\begin{proof}
Let $\pi\in\Sym$. We show that $\pi$ contains
$\fishpattern_d$ if and only if $\Ascbot\pi\not\subseteq\Act_d\pi$.
Initially, suppose that $\pi$ contains an occurrence
$p_ip_{i+1}p_j$ of $\fishpattern_d$. Then $p_i\in\Ascbot\pi$
and $p_i$ is not $d$-active. Thus, $\Ascbot\pi\not\subseteq\Act_d\pi$,
as wanted. On the other hand, suppose that
$\Ascbot\pi\not\subseteq\Act_d\pi$. That is, there is an entry $p_i$
such that $p_i\in\Ascbot\pi$ and $p_i\notin\Act_d\pi$.
Note that $p_i<p_{i+1}$. Further, since $p_i\notin\Act_d\pi$,
by definition of $d$-active site we have $p_i-1=p_j$, for some $j>i$.
Finally, the triple $p_ip_{i+1}p_j$ is an occurrence of
$\fishpattern_d$, finishing the proof.
\end{proof}

Zang and Zhou proved that $\F_0$ coincides with the set of Fishburn
permutations, while $\F_1$ is equal to the set of weak Fishburn
permutations introduced by B\'enyi {\etal}~\cite{BCD:was}.
Further, they showed that $\F_d$ is tree-like in the following sense.

Consider a set $\Pi$ of permutations. As usual, let
$\Pi_n=\Pi\cap\Sym_n$. Say that $\Pi$ is {\em tree-like} if
$\Pi_0=\{\ep\}$ (where $\ep$ is the empty permutation) and, for $n\ge1$,
every $\pi\in\Pi_n$ is obtained by inserting~$n$ into a site of some
$\rho\in\Pi_{n-1}$, called the \emph{parent} of~$\pi$. The spaces
between letters of~$\rho$ into which~$n$ can be inserted are called
the {\em active sites with respect to $\Pi$}, and all other sites
of $\rho$ are said to be {\em inactive}.
Active sites are labeled $1,2,\ldots$ from left to right. The active sites
of $\pi\in\F_d$ are called {\em $d$-active sites} and are the site before
$\pi$ as well as the sites which lie just after a $d$-active element.

Finally, Zang and Zhou generalized the classical encoding of Fishburn
permutations by ascent sequences to $d$-Fishburn permutations and
$d$-ascent sequences; that is, they defined bijections
$$
\Phi_d:\kAscseq{d}\longrightarrow\F_d
$$
by letting, recursively,
\begin{itemize}
\item $\Phi_d(\epsilon)=\epsilon$, and
\item for $n\ge 1$ if $\al=\be a\in\A_{d,n}$ then $\Phi_d(\al)$ is the
result of inserting $n$ into the active site of
$\Phi_d(\be)$ labeled $a$.
\end{itemize}

\subsection{Burge transpose and Fishburn permutations}

The set of \emph{Burge words} is defined as
$$
\Bur_n=\left\lbrace
\binom{u}{\al} \,\big\rvert\, u\in\WI_n,\;\al\in\Cay_n,\;\wDes(u)\subseteq\wDes(\al)
\right\rbrace,
$$
where $\WI_n$ is the subset of $\Cay_n$ consisting of the weakly
increasing Cayley permutations. We define a transposition operation
$T$ on~$\Bur_n$ as follows~\cite{CC:tpb}.
To compute the \emph{Burge transpose} $w^T$ of
$w=\binom{u}{\al}\in\Bur_n$, turn each column of $w$ upside down
and then sort the columns in ascending order with respect to the
top entry, breaking ties by sorting in weakly decreasing order with
respect to the bottom entry. Observe that $T$ is an involution
on $\Bur_n$. Now, let $\identity_n=12\ldots n$ be the identity
permutation. Since $\identity_n$ has no weak descents,
$\binom{\identity_n}{\al}$ is a Burge word for every
$\al\in\Cay_n$. Thus, for any $\al\in\Cay_n$, we can always
pick $\identity_n$ as the top row, and we get a
map $\burget:\Cay_n\to\Sym_n$, defined by
$$
\binom{\identity}{\al}^{\!T}=\binom{\sort(\al)}{\burget(\al)},
$$
for any $\al\in\Cay$, where $\sort(\al)$ is obtained by sorting the
entries of $\al$ in weakly increasing order.
If $\pi\in\Sym$ is a permutation, then $\burget(\pi)=\pi^{-1}$
(and thus $\burget$ is surjective).
Note that the map $\burget$ was originally~\cite{CC:tpb} denoted
by the letter $\ga$.

One of the main advantages of modified ascent sequences and the
Burge transpose is that they give a non-recursive description
of the bijection~$\Phi_0:\kAscseq{0}\to\F_0$.
Indeed~\cite[Corollary 9]{BMCDK:2fp}, if $\alh=\khat{0}(\al)$ is the
modified ascent sequence of~$\al$, then $\burget(\alh)$ is the
Fishburn permutation corresponding to~$\al$ under~$\Phi_0$. With
the ascent sequence $\al=121242232$ of the example before Lemma~\ref{dasc_iff_invseq}, we obtain $\alh=141252232$ and
\begin{align*}
\binom{\identity}{\alh}^{\!T}
&= \binom{1\ 2\ 3\ 4\ 5\ 6\ 7\ 8\ 9}{1\ 4\ 1\ 2\ 5\ 2\ 2\ 3\ 2}^{\!T} \\[1ex]
&= \binom{1\ 1\ 2\ 2\ 2\ 2\ 3\ 4\ 5}{3\ 1\ 9\ 7\ 6\ 4\ 8\ 2\ 5} =
\binom{\sort(\alh)}{\burget(\alh)},
\end{align*}
where $\burget(\alh)=\Phi_0(\al)$.

We wish to use the $d$-hat map to generalize the above construction
to every $d\ge 0$. That is, we shall prove that the diagram
\begin{equation}\label{triangle_any_d}
\begin{tikzpicture}[scale=0.4, baseline=20.5pt]
  \matrix (m) [matrix of math nodes,row sep=3.5em,column sep=7em,minimum width=2em]
  {
    \kAscseq{d} & \F_d  \\
            & \kModasc{d}    \\
  };
  \path[-stealth, semithick]
  (m-1-1) edge node [above, yshift=2pt] {$\Phi_d$} (m-1-2)
  (m-1-1) edge node [below, xshift=-10pt]{$\dhat$} (m-2-2)
  (m-2-2) edge node [right] {$\burget$} (m-1-2);
\end{tikzpicture}
\end{equation}
commutes for every $d\ge 0$ and that all the arrows are size-preserving
bijections. To this end, it will be convenient to let
$$
\img_d=\burget\bigl(\dhat(\dA)\bigr).
$$
Our proof of the commutativity of diagram~\eqref{triangle_any_d} proceeds
in the following steps: First we show that $\burget$ is injective so that
the composition with $\dhat$ is an injective map. Next we demonstrate
that~$\img_d$ has a tree-like structure and describe the active sites
of its permutations. As a corollary, we obtain $\img_d=\F_d$.
The equality $\Phi_d=\burget\circ\,\dhat$ then follows by showing
that $\burget\circ\,\dhat$ has a recursive description that is identical to
the one given for~$\Phi_d$ in terms of active sites.

Let us start by proving that~$\burget$ is injective.
For the rest of this section, given a $d$-ascent sequence~$\al$,
we let~$\alh=\dhat(\al)$ denote the $d$-hat of  $\al$.

\begin{proposition}\label{t-inj}
For all $d,n\ge0$, the map $\burget:\kModasc{d,n}\ra\Sym_n$ is injective.
\end{proposition}
\bprf
We fix $d\ge 0$ and use induction on~$n$. Let us set up the following
notation for the rest of the proof. We will consider two $d$-ascent
sequences~$\al$ and~$\omega$ in~$\kAscseq{d,n}$, where
$$
\al=\be a
\quad\text{and}\quad
\omega=\tau w,
$$
for some~$\be$ and $\tau$ in~$\kAscseq{d,n-1}$. We also let~$a$, $b$, $w$ and~$t$
denote the last letter of~$\al$, $\be$, $\omega$ and~$\tau$, respectively. 
By definition of $\dhat$, we have
\begin{align*}
\alh&=\begin{cases}
\beh a  & \text{if $a\le b-d$};\\
\beh^+a & \text{if $a > b-d$};\\
\end{cases}\\
\shortintertext{and}
\omh&=\begin{cases}
\tah w  & \text{if $w\le t-d$};\\
\tah^+w & \text{if $w>t-d$}.\\
\end{cases}
\end{align*}
We assume~$\alh\neq\omh$ and prove that~$\burget(\alh)\neq\burget(\omh)$.

Assume first that $\be\neq\tau$. Since the map $\dhat$ is injective
by Corollary~\ref{dhat_is_bij}, we have $\beh\neq\tah$ and
by induction, $\burget(\beh)\neq\burget(\tah)$. 
Note that sorting $\binom{\id}{\beh}^T$  and $\binom{\id}{\beh^+}^T$
yields the same permutation in the lower row and similarly for $\tah$
and $\tah^+$. So, whichever case of the $\dhat$ map we are in, we will
be placing~$n$ into the two different permutations $\burget(\beh)$
and $\burget(\tah)$ to compute  $\burget(\alh)$ and $\burget(\omh)$.
This must result in distinct permutations, as desired.

Now assume that $\be=\tau$. Since $\al\neq\omega$ it must be that
$a\neq w$ and we can assume, without loss of generality, that $a<w$.
If $a\le b-d$, then $\binom{\id}{\alh}^T$ is computed from
$\binom{\id}{\beh}^T$ by placing the column $\binom{a}{n}$ at the
beginning of the list of columns with top entry~$a$. Note that such
columns must exist because of the given inequality.
If $a>b-d$ then $\binom{\id}{\alh}^T$ is computed from
$\binom{\id}{\beh^+}^T$ by inserting the column $\binom{a}{n}$ between
the columns with top entry~$a-1$ and those with top entry~$a+1$ (if
there are any of the latter). In this case there will be no other columns
with top entry~$a$. It is now a simple matter of checking to show that
in all possible cases the fact that~$a<w$ will force~$n$ to be
in~$\burget(\alh)$ strictly to the left of its appearance
in~$\burget(\omh)$. This completes the proof of injectivity.
\eprf

\begin{corollary}\label{burget_bij}
For every $d\ge 0$, the map $\burget\circ\,\dhat:\A_d\ra\img_d$ is
a bijection.
\end{corollary}
\begin{proof}
Our claim follows directly from Corollary~\ref{dhat_is_bij},
Proposition~\ref{t-inj} and the definition of $\img_d$.
\end{proof}

\begin{lemma}\label{imgd_treelike}
For all $d\ge0$, the set $\img_d$ is tree-like.
\end{lemma}
\begin{proof}
Clearly $\ep\in\img_{d,0}$, so let $n\ge1$.
Pick $\pi\in\img_{d,n}$ and suppose $\pi=\burget(\alh)$,
for some $\al\in\dA$. Suppose that $\al=\be a$ and consider
$\rho=\burget(\beh)$. We claim that~$\pi$ is obtained
by inserting $n$ in some site of $\rho$ which will prove the theorem.
As usual, there are two cases depending on whether $\alh=\beh a$ or
$\alh=\beh^+ a$.

Suppose first that $\alh=\beh a$. As in the proof of
Proposition~\ref{t-inj}, $\binom{\id}{\alh}^T$ is obtained from
$\binom{\id}{\beh}^T$ by inserting the column $\binom{a}{n}$ at
the beginning of the columns with top entry~$a$. This means that
$\pi=\burget(\alh)$ is obtained from $\rho=\burget(\beh)$ by
inserting~$n$ in the corresponding site, which proves the claim
in this case.

Consider the second case. Here $\binom{\id}{\alh}^T$ is obtained
from $\binom{\id}{\beh^+}^T$ by inserting the column $\binom{a}{n}$
between the columns with upper entry~$a-1$ and those with upper
entry $a+1$, there being no columns with upper entry $a$.
But since $\beh$ and $\beh^+$ are order isomorphic, it follows again
that insertion of $n$ in the corresponding site of $\rho$ yields~$\pi$.
\end{proof}

Since $\img_d$ is tree-like, it is natural to want a description of
the active sites of $\rho\in\img_d$. By Corollary~\ref{burget_bij},
the map $\burget\circ\,\dhat:\A_d\ra\img_d$ is a bijection. 
If $\be=(\burget\circ\,\dhat)^{-1}(\rho)$ then one could
consider all $d$-ascent sequences of the form $\al=\be a$ for
$1\le a \le \dasc\be+1$.
Computing the permutations $\pi=\burget(\alh)$ for each
such $\al$ and comparing with $\rho$ would accomplish this task. But it would
be nice to have a description of the active sites which can be read off from
the permutation itself as one would do for pattern-avoidance classes.
To this end, given $\rho\in\img_d$, let $\be\in\dA$ be its preimage. Now
$$
\binom{\id}{\beh}^T = \binom{\sort(\beh)}{\rho}.
$$
We shall use the active sites of $\rho$ to define a labeling of the sites
of $\sort(\beh)$ by letting a site of $\sort(\beh)$ be active if and only
if the corresponding site of $\rho$ is active.
In this context, active sites refer to the tree-like structure of $\img_d$
established in Lemma~\ref{imgd_treelike}, and should not be confused
with the active sites with respect to~$\F_d$.  We will first describe
the active sites of $\sort(\beh)$. To do so we need the concept of a
{\em run} in a sequence which is a maximum factor (subsequence of
consecutive elements) consisting of equal elements.

\begin{lemma}\label{sites_sortalh}
Suppose $\be\in\kAscseq{d,n-1}$.
\ben
\item[$(a)$] The elements of the runs of $\sort(\beh)$ are
$1,2,\ldots,\dasc\be$ from left to right.
\item[$(b)$] The active sites of $\sort(\beh)$ are the sites before,
after, or between its runs.
\item[$(c)$] The number of active sites of $\burget(\beh)$ is equal
to $\dasc\be+1$.
\een
\end{lemma}
\begin{proof}
$(a)$ This follows directly from Proposition~\ref{dhat_is_cayley}.

\medskip

$(b)$ Suppose $\al=\be a$ where $1\le a\le \dasc\be + 1$.  
If $a$ does not create a $d$-ascent then $\binom{\id}{\alh}^T$
is obtained from $\binom{\id}{\beh}^T$ by inserting the column
$\binom{a}{n}$ at the beginning of the run of $a$'s in $\sort(\beh)$.
This will be  the site before $\sort(\beh)$ or the site between the run
of $(a-1)$'s and the run of $a$'s.
Now suppose $a$ does cause a $d$-ascent so that $\alh=\beh^+a$.  
Note that the runs of $\sort(\beh$) and $\sort(\beh^+)$ are the same
except that the entries in the latter which are greater than or equal
to~$a$ have been increased by one.
Now $\binom{\id}{\alh}^T$ is obtained from $\binom{\id}{\beh}^T$
by inserting the column $\binom{a}{n}$ after the run of $(a-1)$'s in
$\sort(\beh^+)$. So this will either be between the runs for $a-1$
and $a+1$ or at the end.  This shows that the sites before, after,
or between the runs are indeed active.

To see that these are the only active sites, note that $|\img_{d,n}|$ is
the sum of the number of active sites over all elements of $\img_{d,n-1}$.
Since $\img_{d,n}$ is in bijection with $\A_{d,n}$ we have that
$|\img_{d,n}|$ is also the sum of $\dasc\be+1$ over all $\be\in A_{d,n-1}$.
But in the previous paragraph we showed that there are at least $\dasc\be+1$
active sites in every $\sort(\beh)$. Since the two sums are equal, we must
have exactly $\dasc\be+1$ active sites in every $\sort(\beh)$.
Thus there can be no others.

$(c)$ From Item~$(b)$, the number of active sites of $\burget(\beh)$
is equal to one plus the number of runs of~$\beh$. Our claim follows
immediately since there are exactly $\dasc\be$ runs by Item~$(a)$.
\end{proof}

Next we show that the last letter of a $d$-ascent sequence determines
the active site where the maximum of the corresponding permutation
in~$\img_d$ is inserted.

\begin{lemma}\label{act_sites_pos}
Let $d\ge 0$. Let $\al\in\kAscseq{d,n}$ and let
$\pi=\burget(\alh)\in\img_{d,n}$.
Then $\pi$ is obtained by inserting~$n$ in the $a$th active site of
its parent, where~$a$ is the last letter of~$\al$.
\end{lemma}
\begin{proof}
Suppose that $\al=\be a$, for some $\be\in\kAscseq{d,n-1}$.
As observed in the proof of Item~$(b)$ of Lemma~\ref{sites_sortalh},
the active sites of~$\sort(\beh)$ are the sites before, after, or between
its runs. Since the column $\binom{a}{n}$ is inserted at the beginning
of the run of $a$'s in $\sort(\beh)$, or after the last run if no run
of $a$'s exists, it follows immediately that~$n$ is inserted in the
$a$th active site of its parent.
\end{proof}

We now wish to express the active sites of $\pi\in\img_{d,n}$ in terms
of its parent $\rho\in\img_{d,n-1}$.  We will call the sites of $\rho$
which remain between the same two elements in $\pi$ {\em common}.
In addition, there will be two new sites before and after $n$ in $\pi$.
The following criterion is similar to the one~\cite{DS:das} for the
avoidance class of the bivincular pattern
$\si_d=(d+2)|(d+3)12\dots d \overline{(d+1)}$.

\begin{lemma}\label{actsites_imgd}
Suppose $\pi\in\img_{d,n}$ has parent $\rho\in\img_{d,n-1}$.
Then each common site is either active in both $\pi$ and $\rho$ or
inactive in both.  Also, the site before $n$ is always active in $\pi$.
For the site after $n$, let $s$ and $t$ be the number of active sites
before $n$ in~$\pi$ and before $n-1$ in $\rho$, respectively. Then the
site after $n$ is active if and only if 
$$s>t-d.$$
\end{lemma}
\begin{proof}
Let $\al=(\burget\circ\,\dhat)^{-1}(\pi)=\be a$ and let $\be$ have
last element $b$. From the active sites of $\rho$ we can determine
$\sort(\beh)$. More precisely, from Lemma~\ref{sites_sortalh} one can
construct $\sort(\beh)$ by filling in the elements between the $i$th
and $(i+1)$st active sites with $i$'s for each $i\ge1$. Moreover, by
Lemma~\ref{act_sites_pos} the number of active sites
before~$n-1$ is the last letter of~$\be$.

Now consider what happens when the column $\binom{a}{n}$ is added to 
$\binom{\id}{\beh}^T$. Again we see from the proof of
Lemma~\ref{sites_sortalh}, that wherever this column is inserted, it
becomes the beginning of a run of $a$'s. Now using Item~$(b)$ of the
lemma, we see that all the common sites retain their character and that
the site to the left of~$n$ must be active.

Finally, look at the site to the right of~$n$. From the definition of~$s$
and~$t$ as well as the observation at the end of the first paragraph of
this proof, we have $s=a$ and $t=b$. Furthermore, since we only count active
sites before~$n$, we can determine~$s$ just from knowing the sites of~$\rho$
and the position of~$n$ in~$\pi$. So if $s\le t-d$ then $a\le b-d$ and~$a$
does not create a $d$-ascent. It follows that $\binom{a}{n}$ is placed at
the beginning of run of other~$a$'s. So, the site to its right will not
be active since it does not begin a run. On the other hand, if $s>t-d$
then a similar argument shows that the column is inserted as a run of~$a$'s
having only one element. This forces the site to its right to be active
and finishes the proof.
\end{proof}

To prove that $\img_d=\F_d$, we relate active sites with respect to
$\img_d$ with active sites with respect to $\F_d$. To avoid confusion,
we will call a site \emph{$\F_d$-active} if it is active with respect
to $\F_d$, and \emph{$\img_d$-active} if it is active with respect
to $\img_d$.
We will also need the following lemma by Zang and Zhou.

\begin{lemma}\cite[Lemma 2.5]{ZZ:dperm}\label{lemmaZZ}
Let $d\ge 0$ and $n\ge 1$. Let $\pi\in\Sym_n$ and let $\rho$ be obtained
by removing~$n$ from $\pi$. Then $\pi\in\F_{d,n}$ if and only if
$\rho\in\F_{d,n-1}$ and~$n$ is placed before~$\rho$ or after some
$d$-active element of~$\rho$.
\hfill{\qed}
\end{lemma}

By Lemma~\ref{lemmaZZ}, the $\F_d$-active sites of $\rho\in\F_d$
are precisely those positions that follow a $d$-active element of~$\rho$,
together with the position before the leftmost entry.

\begin{theorem}\label{act_iff}
For any $d,n\ge 0$,
$$
F_{d,n}=\img_{d,n}.
$$
Furthermore, a site of $\pi\in\F_{d,n}=\img_{d,n}$ is $\F_d$-active
if and only if it is $\img_d$-active.
\end{theorem}
\begin{proof}
We use induction on~$n$, where the claim holds for $n\le 1$.
Let $n\ge 2$ and assume that $\F_{d,n-1}=\img_{d,n-1}$.
By induction, given $\rho\in\F_{d,n-1}=\img_{d,n-1}$, a site of~$\rho$
is $\F_d$-active if and only if it is $\img_d$-active.
Since both $\img_d$ and $\F_d$ are tree-like, by Lemmas~\ref{imgd_treelike}
and~\ref{lemmaZZ}, respectively, the equality $\F_{d,n}=\img_{d,n}$
follows immediately.

Let us now consider a permutation $\pi\in\F_{d,n}=\img_{d,n}$.
We have to show that a site of~$\pi$ is $\F_d$-active if and only if
it is $\img_d$-active.
The site before the leftmost entry is active in both cases by item~(b)
of Lemma~\ref{sites_sortalh} and by Lemma~\ref{lemmaZZ}.
Now, let $\rho\in\F_{d,n-1}=\img_{d,n-1}$ be the parent of~$\pi$.
By Lemma~\ref{actsites_imgd} each common site is $\img_d$-active
in~$\pi$ if and only if it is $\img_d$-active in~$\rho$; and
the new site before~$n$ is $\img_d$-active.
Similarly, by definition of $d$-active entry and Lemma~\ref{lemmaZZ},
each of these sites is $\F_d$-active in~$\pi$ if and only if it is
$\F_d$-active in~$\rho$, and $n$ is always placed in an $\F_d$-active site
which is directly after an $\F_d$-active element.
Since by induction $\F_d$-active and $\img_d$-active sites of~$\rho$
coincide, the desired claim holds for every common site, as well as for
the new site before~$n$.

To finish the proof of the theorem, we only need to consider
the new site after~$n$. Using the same notation as in
Lemma~\ref{actsites_imgd}, let~$s$ and~$t$ be the number of
$\img_d$-active sites before~$n$ in~$\pi$ and before~$n-1$ in~$\rho$,
respectively. By this lemma, the site after~$n$ is $\img_d$-active
if and only if $s>t-d$. If~$n$ appears to the right of $n-1$ in~$\pi$,
then~$n$ is $\F_d$-active. Moreover, we have $s\ge t+1$ since the site
before~$n$ is $\img_d$-active. Thus
$$
s\ge t+1>t\ge t-d
$$
and the site after~$n$ is $\img_d$-active, as desired.
On the other hand, suppose that~$n$ appears to the left of $n-1$.
Write
$$
\pi=g_1\ldots g_i\;n\;g_{i+1}\ldots g_j\;(n-1)\;g_{j+1}\ldots g_{n-1},
$$
for some $i\le j$. 
We have
\begin{align*}
t &=\text{$\#$ $\img_d$-active sites before $(n-1)$ in $\rho$}\\
  &=\text{$\#$ $\F_d$-active sites before $(n-1)$ in $\rho$}\\
\shortintertext{by induction, and}
s &=\text{$\#$ $\img_d$-active sites before $n$ in $\pi$}\\
&=\text{$\#$ $\img_d$-active sites before $g_{i+1}$ in $\rho$}\\
&=\text{$\#$ $\F_d$-active sites before $g_{i+1}$ in $\rho$}\\
\end{align*}
where the last step is again by induction.  Therefore,
\begin{align*}
t-s
&=\text{$\#$ $\F_d$-active sites between $g_{i+1}$ and $n-1$ in $\rho$}\\
&=\text{$\#$ $\F_d$-active entries in $g_{i+1}\ldots g_j$ in $\rho$},\\
\end{align*}
where at the last step we used Lemma~\ref{lemmaZZ}.
Finally, by Lemma~\ref{actsites_imgd}, the site after $n$ is $\img_d$-active if and only if
$s> t-d$.  Rearranging terms gives $t-s < d$ which is equivalent to $n$ being a $d$-active element by the definition of $d$-active entries. In turn, this is equivalent to the site after $n$ being $\F_d$-active by Lemma~\ref{lemmaZZ}. This completes the proof.
\end{proof}

\begin{theorem}\label{dfish_thm}
For any $d\ge 0$,
$$
\Phi_d=\burget\circ\,\dhat.
$$
\end{theorem}
\begin{proof}
We have established in Theorem~\ref{act_iff} that the maps~$\Phi_d$
and~$\burget\circ\,\dhat$ have the same image $\F_d=\img_d$.
Let us prove inductively that $\Phi_d=\burget\circ\,\dhat$.
Let~$\al=\be a\in\kAscseq{d,n}$, where~$\be\in\kAscseq{d,n-1}$ and
$1\le a\le 1+\asc\be$. By induction, we have
$$
\Phi_d(\be)=\burget(\dhat(\be))=:\rho.
$$
Again by Theorem~\ref{act_iff}, a site of~$\rho$ is $\F_d$-active if
and only if it is $\img_d$-active. Moreover, the last letter of~$\al$
determines the label of the active site where~$n$ is inserted both
under~$\Phi_d$, by definition, and under~$\burget\circ\,\dhat$,
by Lemma~\ref{act_sites_pos}. Thus $\Phi_d(\al)=\burget(\dhat(\al))$,
finishing the proof.
\end{proof}

\section{Pattern avoidance in \texorpdfstring{$\F_d$}{Fd}}\label{sec_pattav}

The introduction and characterization of the $d$-Fishburn permutations
opens the door to pattern avoidance results parameterized by $d$. As an
illustration, we shall study one such instance in some depth, namely
the case of $d$-Fishburn permutations avoiding the classical pattern $213$.
First, recall the bivincular pattern
\[
  \si_d=(d+2)|(d+3)12\dots d \overline{(d+1)}.
\]

Zang and Zhou~\cite[Theorem 2.4]{ZZ:dperm} proved that
\begin{equation}
\label{FSs}
\F_d\subseteq\Sym(\si_d)
\end{equation}
for every $d\ge 0$, where for $d=0,1$
equality holds.

\begin{proposition}
We have $\F_d(213)=\Sym(\si_d,213)$.
\end{proposition}
\begin{proof}
The inclusion $\F_d(213)\subseteq\Sym(\si_d,213)$ follows 
from~\eqref{FSs}. For the same reason, if $d\le 1$ we obtain the
desired equality.
Now let $d\ge 2$. We shall prove the remaining inclusion
$\Sym(\si_d,213)\subseteq\F_d(213)$. Let $\pi\in\Sym_d(\si,213)$.
For a contradiction, suppose that $\pi\notin\F_d$. That is,
$\pi$ contains an occurrence $p_ip_{i+1}p_j$ of~$\fishpattern$
where $p_i$ is not a $d$-active element.
Since~$p_i$ is not $d$-active, there are at least~$d$ entries
$p_{u_1},\dots,p_{u_d}$, $u_1<u_2<\cdots<u_d$, between~$p_{i+1}$
and~$p_j$ that are smaller than~$p_j$ (and $d$-active).
Further, since $d\ge 2$, these must be in increasing order or else
they would create an occurrence of~$213$ with~$p_j$.
Thus we have obtained an occurrence $p_ip_{i+1}p_{u_1}\ldots p_{u_d}p_j$
of~$\si_{d}$, which is impossible.
\end{proof}

In order to enumerate $\F_d(213)$, we show that~$\Sym_{n}(\si_d,213)$
is in bijection with the set of Dyck paths of semilength~$n$ that do
not contain~$\D\D\U^{d+1}$ as a factor.
Let us start by defining a bijection~$\phi$ from~$\Sym_n(213)$ to Dyck
paths of semilength~$n$. It is simply a tilted version of what is
sometimes called~\cite{ClKi2008} the \emph{standard bijection} from
$132$-avoiding permutations to Dyck paths. Any non-empty
permutation~$\pi\in\Sym(213)$ decomposes uniquely as
$$
\pi=p_1 L R,
$$
where all the entries in~$L$ are larger than~$p_1$, and all the entries
in~$R$ are smaller than~$p_1$. Then~$\phi$ is defined recursively
by mapping the empty permutation to the empty path and letting
$$
\phi(\pi)=\phi(p_1 L R)=\U\phi(L)\D\phi(R),
$$
where here we abuse notation and use the same letter~$L$ for the
permutation that is order isomorphic to~$L$. Under the bijection~$\phi$,
the value of the first letter determines the first return to the
$x$-axis.

We show that~$\phi$ restricts to a bijection from~$\Sym(\si_d,213)$ to
Dyck paths avoiding $\D\D\U^{d+1}$ as a factor, for every $d\ge 0$.
First a lemma whose easy proof is omitted.

\begin{lemma}\label{incr_prefix}
Let $\pi\in\Sym_n(213)$ and let $\rho=\phi(\pi)$
be the corresponding Dyck path. Then

\eqqed{
p_1<p_2<\cdots<p_k
\;\iff\;
\U^k\text{ is a prefix of $\rho$.}
}
\end{lemma}

\begin{lemma}\label{sigma_iff_factor}
Let $\pi\in\Sym_n(213)$ and let $\rho=\phi(\pi)$
be the corresponding Dyck path. Then, for any $d\ge 0$,
$$
\pi\text{ contains }\si_{d}
\;\iff\;
\D\D\U^{d+1}\text{ is a factor of $\rho$.}
$$
\end{lemma}
\begin{proof}
We use induction on~$n$, where $n=0$ and $n=1$ are trivial. Assume our
claim holds for $n-1$ where $n\ge 2$ and let $\pi=p_1LR\in\Sym_n(213)$. Initially, suppose that~$\pi$ contains an occurrence
$p_ip_{i+1}p_{u_1}\ldots p_{u_d}p_j$ of~$\si_{d}$. If either
$p_j\in L$ or $p_i\in R$, then we can conclude that~$\rho$ contains
a factor $\D\D\U^{d+1}$ by induction.
Otherwise, since entries in~$L$ are larger than entries in~$R$, it must
necessarily be that $p_{i+1}\in L$ while $p_{u_1}$ is contained in~$R$.
Moreover, since $p_j=p_i-1$, we have $i=1$ and~$p_j$ is the largest
entry in~$R$.
Now, since~$L$ is not empty, the path~$\U\phi(L)\D$ ends with~$\D\D$.
Furthermore, since~$\pi$ avoids~$213$, all the entries preceding~$p_j$
in~$R$ are in increasing order. Taking~$p_j$ into account, (at least)
the first~$d+1$ entries of~$R$ are in increasing order. Using
Lemma~\ref{incr_prefix}, it follows that~$\phi(R)$ starts with~$\U^{d+1}$.
Hence the last two steps of~$\U\phi(L)\D$ form a factor~$\D\D\U^{d+1}$
with the first~$d+1$ steps of~$\phi(R)$, as wanted.

On the other hand, suppose that~$\rho$ contains a factor $\D\D\U^{d+1}$.
We will show that~$\pi$ contains~$\si_{d}$. Similarly to the argument in
the previous paragraph, if the whole factor $\D\D\U^{d+1}$ is
contained in either~$\phi(L)$ or~$\phi(R)$, then we can conclude the proof by
induction. Otherwise, it must be that the last two steps
of~$\U\phi(L)\D$ are~$\D\D$ and the first~$d+1$ steps of~$\phi(R)$
are~$\U^{d+1}$. Since~$\phi(L)$ is not empty, we have~$p_1<p_2$. 
Using Lemma~\ref{incr_prefix} once again, we have that the
first $d+1$ entries of~$R$, say $p_{u_1},\dots,p_{u_d},p_{u_{d+1}}$,
are in increasing order. Finally, the maximum entry of~$R$ is equal
to $p_1-1$, and we obtain the desired occurrence
$p_1p_2p_{u_1}\ldots p_{u_d}(p_i-1)$ of~$\si_{d}$ in~$\pi$.
\end{proof}

For any fixed $d\geq 0$, we shall derive a generating function for the
numbers $\#\F_{d,n}(213)$. By the preceding proposition we can achieve
this by counting Dyck paths having no $\D\D\U^{d+1}$ factor. In fact, we
shall derive a generating function for the distribution of the number of
$\D\D\U^{d+1}$ factors over Dyck paths. Let us start with the case
$d=0$. In the spirit of the cluster method~\cite{GJ:cluster,Wang:phd},
consider Dyck paths in which a subset of the $\D\D\U$ factors have been
marked. For instance,
\[
  \rho = \U\U\D\U\underline{\D\D\U}\U\U\D\D\U\underline{\D\D\U}\D
\]
has three $\D\D\U$ factors, two of which have been marked
(underlined). Let us encode $\rho$ as a word $\rho'$ over the alphabet
$\{\U,\D,\D'\}$ by replacing each marked $\D\D\U$ factor with a
$\D'$. In our example we have
\[
  \rho' = \U\U\D\U\D'\U\U\D\D\U\D'\D.
\]
Note that $\rho'$ represents a marked Dyck path if and only if $\rho'$
itself is a Dyck path, when interpreting $\D'$ as $\D$, and the height
at which any $\D'$ step starts is at least two.

Let $\Path_0\in\QQ\langle\U,\D,\D'\rangle$ be the formal sum of
Dyck paths with two sorts of down steps, $\D$ and $\D'$. By the
usual first return decomposition $\Path_0$ satisfies
\[
  \Path_0 = 1 + \U \Path_0 \D \Path_0 + \U \Path_0 \D' \Path_0.
\]
Let $\Q_0\in\QQ\langle\U,\D,\D'\rangle$ be the formal sum of the subset
of the paths encoded in $\Path_0$ defined by requiring that the height
at which any $\D'$ step starts is at least two. Then
\[
  \Q_0 = (\U\Path_0\D)^*,
\]
where we use the (Kleene star) convention
$\mathcal{F}^*=1+\mathcal{F}+\mathcal{F}^2+\cdots$.  Define the map
$\varphi:\QQ\langle \U,\D,\D'\rangle\to\QQ[q,x]$ by
$\U\mapsto x$, $\D\mapsto 1$, $\D'\mapsto qx$ and extending by
linearity. Now, letting $P_0(q,x)=\varphi(\Path_0)$ 
and
$Q_0(q,x)=\varphi(\Q_0)$, we get the functional equations:
\begin{align*}
  P_0(q,x) &= 1 + x P_0(q,x)^2 + qx^2 P_0(q,x)^2;\\
  Q_0(q,x) &= 1/(1-xP_0(q,x)).
\end{align*}
Note that
\[
  \sum_{\rho}\, (1+q)^{\D\D\U(\rho)}\, x^{|\rho|} \,=\,  Q_0(q, x),
\]
where the sum ranges over all Dyck paths, $|\rho|$ is the semilength of
$\rho$, and $\D\D\U(\rho)$ is short for the number of $\D\D\U$ factors
in $\rho$. Indeed, the power series $Q_0(q,x)$ counts Dyck paths with
respect to semilength and number of marked $\D\D\U$ factors, but so does
the left-hand side: For each of the $\D\D\U$ factors there is a choice to
be made, mark it (with a $q$) or leave it unmarked. Thus,
\begin{equation}\label{Q_0}
  Q_0(q-1, x) \,=\, \sum_{\rho}\, q^{\D\D\U(\rho)}\, x^{|\rho|}
\end{equation}
is the generating function we seek.
In particular, $Q_0(-1, x)$ is
the generating function for Dyck paths with no $\D\D\U$ factors.

A similar analysis applies when $d\geq 1$. In this case we consider Dyck
paths $\rho$ in which a subset of the $\D\D\U^{d+1}$ factors are marked,
and we encode such a path by a word $\rho'$ over the alphabet
$\{\U,\U',\D\}$, where $\U'$ represents a marked $\D\D\U^{d+1}$
factor. In this way, $\rho'$ represents a marked Dyck path if and only
if $\rho'$ itself is a Dyck path, when interpreting $\U'$ as $\U^{d-1}$,
and the height at which any $\U'$ step starts is at least two. As the
reader may have noticed, for the preceding description to make sense in
the special case $d=1$ we need to view $\U^0$ as a level-step and
in this case we are really dealing with Motzkin paths rather than Dyck
paths. However, the equations describing the
resulting language hold uniformly for any $d\geq 1$ and this is the
reason for not separating out $d=1$ as a special case.

Let $\Path_d\in\QQ\langle\U,\U',\D\rangle$  be the formal sum of Dyck
paths with two sorts of up steps, $\U$ and $\U'$, where each $\U'$ can
be thought of representing $\D\D\U^{d+1}$ and thus each such step
contributes $d-1$ to the height of the path. By a simple extension of
the first return decomposition we find that
\[
  \Path_d = 1 + \U \Path_d \D \Path_d + \U' \Path_d (\D \Path_d)^{d-1}.
\]
Let $\Q_d\in\QQ\langle\U,\U',\D\rangle$ be the formal sum of the subset
of the paths encoded in $\Path_d$ defined by requiring that the height
at which any $\U'$ step starts is at least two. Then
\[
  \Q_d = \bigl(\U(\U\Path_d\D)^{*}\D\bigr)^{*}.
\]
Define $\varphi:\QQ\langle\U,\U',\D\rangle \to \QQ[q,x]$ by
$\U\mapsto x$, $\D\mapsto 1$ and $\U'\mapsto qx^{d+1}$. Then, with
$P_d(q,x) = \varphi(\Path_d)$ and $Q_d(q,x) = \varphi(\Q_d)$, we have
\begin{align*}
  P_d(q,x) &= 1 + xP_d(q,x)^2 + qx^{d+1} P_d(q,x)^{d-1};\\
  Q_d(q,x) &= \cfrac{1}{1-\cfrac{x}{1-xP_d(q,x)}}.
\end{align*}
By following the same line of reasoning as were used to demonstrate
identity~\eqref{Q_0} we arrive the following result.

\begin{proposition}\label{factor_distrib}
  For any $d\geq 0$,
  \[
    \sum_{\rho}\, q^{\D\D\U^{d+1}(\rho)}\, x^{|\rho|} \,=\,  Q_d(q-1, x),
  \]
  where the sum ranges over all Dyck paths, $|\rho|$ is the semilength
  of $\rho$, and $\D\D\U^{d+1}(\rho)$ is short for the number of
  $\D\D\U^{d+1}$ factors in $\rho$.
  \hfill{\qed}
\end{proposition}

By combining Lemma~\ref{sigma_iff_factor} and
Proposition~\ref{factor_distrib} we arrive at the desired generating
function for $213$-avoiding $d$-Fishburn permutations.

\begin{theorem}\label{Fd213}
  For any $d\geq 0$,
  
  \eqqed{
    \sum_{\pi\in\F_d(213)} x^{|\pi|} \,=\, Q_d(-1, x).
  }
\end{theorem}

For a fixed small $d$ one can derive an explicit expression for
$Q_d(-1, x)$ by solving the corresponding system of functional
equations. We have done so for $d\leq 2$:
\begin{align*}
  Q_0(-1,x) &\,=\, \frac{1-x}{1-2x}; \\[2ex]
  Q_1(-1,x) &\,=\, \frac{2(1 - x)}{1 - 2x + x^2 + \sqrt{1 - 4x + 2x^2 + x^4}}; \\[2ex]
  Q_2(-1,x) &\,=\, \frac{2(1 - x)}{1 - 2x + 2x^2 + \sqrt{1-4x+ 4x^3}}.
\end{align*}
Since $\F_d(213)=\Sym(\si_d,213)$ and $\#\Sym_n(213)=C_n$,
the $n$th Catalan number, we find that the sequence of series
$\{Q_d(-1,x)\}_{d\geq 0}$ converges to the generating function for the
Catalan numbers:
\[
  \lim_{d\to\infty} Q_d(-1,x) = \frac{2}{1+\sqrt{1-4x}}
\]
The coefficient of $x^n$ in $Q_0(-1,x)$ is $2^{n-1}$ for
$n\geq 1$, and hence one might say that the coefficients in $Q_d(-1,x)$
``interpolate'' between $2^{n-1}$ and $C_n$; in Table~\ref{table-F213}
we list the first few coefficients of $Q_d(-1,x)$ for $d\leq 5$.
\begin{table}
\[
\begin{array}{c|rrrrrrrrrrrrr}
  d\hspace{2pt}\backslash\hspace{1pt} n \rule[-0.9ex]{0pt}{0pt}
  & 0 & 1 & 2 & 3 & 4 & 5 & 6 & 7 & 8 & 9 & 10 & 11 & 12 \\
  \hline \rule{0pt}{2.6ex}
  0 & 1 & 1 & 2 & 4 & 8 & 16 & 32 & 64 & 128 & 256 & 512 & 1024 & 2048  \\
  1 & 1 & 1 & 2 & 5 & 13 & 35 & 97 & 275 & 794 & 2327 & 6905 & 20705 & 62642 \\
  2 & 1 & 1 & 2 & 5 & 14 & 41 & 124 & 384 & 1212 & 3885 & 12614 & 41400 & 137132 \\
  3 & 1 & 1 & 2 & 5 & 14 & 42 & 131 & 420 & 1375 & 4576 & 15434 & 52639 & 181230 \\
  4 & 1 & 1 & 2 & 5 & 14 & 42 & 132 & 428 & 1420 & 4796 & 16432 & 56966 & 199448 \\
  5 & 1 & 1 & 2 & 5 & 14 & 42 & 132 & 429 & 1429 & 4851 & 16718 & 58331 & 205632
\end{array}
\]
\caption{Number of $213$-avoiding $d$-Fishburn permutations of length $n$.}
\label{table-F213}
\end{table}

The transport of patterns between Fishburn permutations and modified
ascent sequences developed by the first two authors~\cite{CC:tpb}
applies to $d$-Fishburn permutations and modified $d$-ascent sequences as well.
Call two Cayley permutations~$\al$ and~$\be$  {\em equivalent} if
$\burget(\al)=\burget(\be)$, and let $[\Cay]$ denotes the set of equivalence classes over~$\Cay$ defined this way. Moreover, an element~$[\al]$
of~$[\Cay]$ contains~$[\rho]$ if~$\al'$ contains~$\rho'$ for
some~$\al'\in[\al]$ and~$\rho'\in[\rho]$.
We denote by
$[\Cay][\rho]$ the
set of classes that avoid~$[\rho]$. 
By the transport theorem on equivalence classes of Cayley
permutations~\cite[Theorem~4.9]{CC:tpb}, the Burge transpose induces
a bijection
$$
\burget:[\Cay][\rho]\to\Sym\bigl(\burget(\rho)\bigr).
$$
Since each equivalence
class contains at most one modified ascent sequence
and~$\burget(\kModasc{0})=\F_0$, we obtain a size-preserving bijection
$$
\burget:\kModasc{0}[\rho]\to\F_0\bigl(\burget(\rho)\bigr),
$$
where $\kModasc{0}[\rho]$ is the set of modified ascent sequences
avoiding every pattern in~$[\rho]$. Equivalently~\cite[Theorem~5.1]{CC:tpb},
for every permutation~$\tau$ we have a size-preserving bijection
$$
\burget:\kModasc{0}(B_{\tau})\to\F_0(\tau),
$$
where $B_{\tau}=[\tau^{-1}]$ is the Fishburn basis of $\tau$.
A constructive procedure to compute~$B_{\tau}$ was given in the
same reference.

Now we have proved in Proposition~\ref{t-inj} that the map~$\burget$
is injective on $\dModasc$ for every $d\ge 0$. Therefore, each
equivalence class of Cayley permutations contains at most one
modified $d$-ascent sequence. Since $\burget(\dModasc)=\F_d$,
we obtain the following transport theorem.

\begin{theorem}\label{transport-thm}
For any $d\ge 0$ and permutation~$\tau$,
$$
\burget:\dModasc(B_{\tau})\longrightarrow\F_d(\tau)
$$
is a size-preserving bijection, where $B_{\tau}$ is the Fishburn basis
of $\tau$, $\dModasc(B_{\tau})$ is the set of modified $d$-ascent
sequences avoiding every pattern in $B_{\tau}$, and $F_d(\tau)$ is the
set of $d$-Fishburn permutations avoiding~$\tau$. In particular,

\eqqed{
\#\F_{d,n}(\tau)=\#\Modasc_{d,n}(B_{\tau}).
}
\end{theorem}

For instance, $B_{213}=\{112,213\}$ and by combining Theorems~\ref{Fd213}
and \ref{transport-thm} we get the following result.

\begin{corollary}
  For any $d\geq 0$,
  
  \eqqed{
    \sum_{\alpha\in\dModasc(112,213)} x^{|\alpha|} \,=\, Q_d(-1, x).
  }
\end{corollary}

It would be interesting to make a deeper study of pattern avoidance in $d$-Fishburn permutations
and (modified) $d$-ascent sequences.

\section{Final remarks}\label{section_final}

It would be desirable to have a better understanding of $\Modinv$.
Computer calculations show that the first few terms of
the sequence $|\Modinv_n|$, starting from $n=0$, are
$$
1,1,3,10,43,224,1396,10136,84057.
$$
We also recall the open problem from Section~\ref{section_modinv}.
\begin{problem}
	Find a characterization of which Cayley permutations lie in $\Modinv$,
	perhaps similar to that of $\kAscseq{0}$ in equation~\eqref{eq_modasc_char}.
\end{problem}

There are many properties of the bijection $\maxhat$ which remain
to be investigated.
In Section~\ref{sec_subd}, we characterized the image of $\Ascseq_0$
under this map. It is natural to ask which sets of permutations are obtained
by restricting $\maxhat$ to the set $\kAscseq{0}(p)$ of ascent sequences which
avoid a pattern $p$. In this regard, we have several conjectures.
\begin{conjecture}
The map $\maxhat$ restricts to the following bijections.
\begin{enumerate}
    \item $\Ascseq_0(123)\longrightarrow\Sym(123,213)$,
    \item $\Ascseq_0(112)\longrightarrow\Sym(213,312)$,
    \item $\Ascseq_0(121)\longrightarrow\Sym(213,231)$,
    \item $\Ascseq_0(213)\longrightarrow\Sym(213,45123)$.
\end{enumerate}
\end{conjecture}
We note that the enumeration of $\kAscseq{0}(p)$, for
$p\in\{111,211,221,231,312\}$, is currently open.

One could also hope to find analogues of the characterization of
$\maxhat(\Ascseq_0)$ in terms of ir-subdiagonal permutations for
larger $d$.
\begin{question}
    What can we say about $\maxhat(\dA)$, for $d>0$?
    Since $\kAscseq{0}\subseteq\dA$, can we describe $\maxhat(\dA)$
    by a similar notion of subdiagonality?
\end{question}

The approach adopted in Section~\ref{sec_subd} can be generalized as
follows. Let $U\subseteq\I$ be any subset of $\I$. Given any $\al\in U$,
choose uniquely a nonnegative integer $d_{\al}$, with $d_{\al}\ge\mind\al$.
By Proposition~\ref{Hset_disj}, we obtain an injection
\begin{align*}
\{(\al,d_{\al})\}_{\al\in U}\;&\longrightarrow\;\Modinv\\
(\al,d_{\al})\;&\longmapsto\;\khat{d_{\al}}(\al).
\end{align*}
What other choices of $U$ and $d_{\al}$ give interesting examples?
A natural choice consists in using $d_{\al}=\mind\al$. Can we
describe the corresponding subset of $\Modinv$?
Conversely, what sets of permutations $T\subseteq\Sym$ can
be pulled back to interesting sets of pairs
$\{(\al,d_{\al})\}_{\al\in U}$?

\medskip

{\bf Acknowledgment.}  We would like to thank Robin D.~P.~Zhou and two anonymous referees for their careful reading of the paper and for pointing out a number of typographical errors.

\nocite{*}
\bibliographystyle{alpha}
\newcommand{\etalchar}[1]{$^{#1}$}

\end{document}